%% file: waveguidecondition.tex
\documentclass[aps,prx,preprint,groupedaddress]{revtex4-1}

\input latex_defs.tex

\usepackage{amssymb}
\usepackage{amsfonts}
\usepackage[cmex10]{amsmath}
\usepackage{array}
\usepackage{natbib}
\usepackage{lineno}
\usepackage{amssymb,bm}
\usepackage{cases}
\usepackage{epstopdf}
\usepackage{graphics,graphicx,subfigure}
\usepackage{amsthm} 
\begin{document}
\renewcommand{\thesection}{\arabic{section}}
\renewcommand{\thesubsection}{\thesection.\arabic{subsection}}
\renewcommand{\thesubsubsection}{\thesubsection.\arabic{subsubsection}}
\theoremstyle{plain}
\newtheorem{theorem}{Theorem}[section]

\theoremstyle{definition}
\newtheorem{definition}{Definition}[section]
\newtheorem{remark}{Remark}
\newtheorem{lemma}{Lemma}[section]
\newtheorem{corollary}{Corollary}[section]

\numberwithin{equation}{section}
\newcommand{\m}[2]{\ensuremath{\,#1\,}#2}
\newcommand{\mm}[2]{\ensuremath{\:#1\:}#2}
\title{Symmetry of Solutions and Domain-Reduction Finite Element Method for Second-Order Linear Elliptic Dirichlet Boundary Value Problems on Bounded Domains}


\author{ Xianlong Pan and Wei Jiang}
\thanks{Corresponding author: Wei Jiang,  Email: jwmathphy@163.com}
\affiliation{School of Physics and  Mechatronics Engineering, Guizhou Minzu University, Guiyang 550025, China}

\begin{abstract}
Combining classical group theory and partial differential equation theory, this paper investigates the symmetry group $\operatorname{Sym}(u)$ of the unique solution $u$ to the second-order linear elliptic boundary value problem on an $n$-dimensional bounded domain $\Omega$
\begin{subequations}
\begin{numcases}{}
    -\sum_{i,j=1}^{n} a_{ij}(x)u_{x_ix_j} + \sum_{i=1}^{n} b_i(x)u_{x_i} + c(x)u = f(x), & $x\in \Omega$, \nonumber\\
    u(x) = h(x), & $x\in \partial \Omega$.\nonumber
\end{numcases}
\end{subequations}
The following symmetry groups are defined and characterized respectively: the symmetry group $\operatorname{Sym}(A)$ of the second-order coefficient matrix function $A(x)=(a_{ij}(x))_{n\times n}$; the symmetry group $\operatorname{Sym}(b)$ of the first-order coefficient column vector function $b(x)=(b_{1}(x),b_{2}(x),\cdots,b_{n}(x))^{T}$; the symmetry group $\operatorname{Sym}(c)$ of the zero-order coefficient function $c(x)$; the symmetry group $\operatorname{Sym}(f)$ of the internal source function $f(x)$; and the symmetry group $\operatorname{Sym}(h)$ of the boundary source function $h(x)$. This paper rigorously proves that the common symmetry group $\operatorname{Sym}(A)\cap\operatorname{Sym}(b) \cap\operatorname{Sym}(c) \cap\operatorname{Sym}(f) \cap\operatorname{Sym}(h)$ is a subgroup of $\operatorname{Sym}(u)$. In addition, if the common symmetry group contains several mirror symmetry elements, the original second-order linear elliptic boundary value problem on the entire domain $\Omega$ can be reduced to the corresponding boundary value problem on a certain subdomain. It is strictly proven in this paper that the new boundary condition imposed on the boundary of the subdomain is the homogeneous generalized Neumann boundary condition. The linear finite element method is used to numerically solve the second-order linear elliptic boundary value problem on the subdomain, thereby achieving domain reduction and significantly reducing the computational cost. Finally, the correctness of the theoretical results is verified through one theoretical model example and three numerical experiments.
\end{abstract}
\maketitle
\section{Introduction}
\indent Second-order linear elliptic partial differential equations constitute the core mathematical models for describing diverse steady-state equilibrium problems. Their comprehensive theoretical framework forms an indispensable component of modern partial differential equation theory, which enables precise characterization of the distribution laws governing various time-invariant physical fields, including electrostatic fields, magnetostatic fields, steady temperature fields, elastic stress fields, and irrotational steady flow fields. Widely deployed in engineering scenarios such as electromagnetic device design, structural mechanical analysis, and industrial thermal control optimization, these equations serve as pivotal modeling tools bridging mathematical theory and engineering practice.\\
\indent Differential equations serve as the core modeling vehicle for all classical mathematical-physical models, encompassing the Maxwell system governing the evolution of electromagnetic fields, the hydrodynamic systems characterizing fluid motion behaviors, and the Schr?dinger equation describing quantum microscopic states. Their intrinsic laws and solution properties can be analyzed relying on the theoretical framework of partial differential equations\cite{zhou2005}. The theory of Lie groups and Lie algebras constitutes a powerful mathematical tool for symmetry analysis of differential equations, establishing a systematic methodology for symmetry identification, classification and reduction of equations. Among relevant researches, Ovsiannikov constructed a complete framework for group analysis of differential equations and laid the foundational groundwork for this research field \cite{Ovsi1982}. In his seminal monograph, Olver further refined the application framework of Lie group approaches and formulated the currently prevalent theoretical paradigm for group analysis of differential equations, which has been extensively adopted in symmetry investigations of diverse mathematical-physical equations \cite{Olver1993}. As for the research on symmetry of solutions to elliptic equations, classical theories based on the maximum principle and the moving plane method have further enriched symmetry analysis techniques, forming a crucial theoretical basis for exploring the symmetry of solutions to elliptic problems on bounded domains and accelerating the rapid advancement of symmetry theory for elliptic equations \cite{GIDAS1979,FRA2000}.\\
\indent In the field of numerical computation, the concept of leveraging symmetry to reduce computational domains has a long-standing history. Ballisti et al.\cite{bal1982} first applied the theory of finite group representations to electromagnetic field computation problems with symmetric boundaries. Douglas and Mandel established an abstract theoretical framework for domain reduction methods, furnishing a rigorous mathematical foundation for symmetry-based reduction\cite{dou192}. Allgower et al.\cite{all1992} systematically investigated the exploitation of symmetry within the boundary element method, decomposing the original problem into a set of subproblems via equivariant mappings. Bossavit conducted systematic research on utilizing the symmetry inherent to physical problems to decompose linear operator equations within the finite element framework. By introducing group representation theory and noncommutative harmonic analysis, he split the original problem into multiple subproblems defined on symmetric cells, which drastically cuts down computational complexity\cite{bos1986,bos1993}. Lobry and Broche\cite{lob1994,lob1996} further extended group representation theory to boundary element formulations and three-dimensional eddy current problems, addressing geometric symmetries and non-Abelian symmetry groups.
In recent years, Hou, Liu and Zhou proposed a symmetrized two-scale finite element method for partial differential equations admitting symmetric solutions\cite{hou2022}. Taking advantage of the symmetry of solutions over tensor-product meshes, this method reduces high-dimensional finite element approximations into combinations defined on coarse meshes and one-dimensional fine meshes, achieving substantial computational cost savings while retaining asymptotically optimal accuracy. Wang et al. developed a reduced-order algorithm based on scaled boundary finite element discretization for cyclically symmetric structures, which delivers remarkable improvements in computational efficiency for elasticity analyses\cite{wang2023}. Furthermore, in finite element modal analyses of photonic devices, researchers have attained manifold speedups by jointly exploiting point-group spatial symmetries and pseudo-Hermitian spacetime symmetries\cite{wang2024}. These studies demonstrate that symmetry can not only facilitate computational domain reduction, but also guide the construction of multiscale schemes and reduced-order models.

\indent The second-order linear elliptic Dirichlet boundary value problem defined on bounded domains represents a canonical model in computational mathematics and engineering numerical simulation, and the development of its high-efficiency numerical solvers has long remained a prominent research focus among scholars worldwide\cite{wang2015}. Current research on elliptic boundary value problems predominantly centers on rigorous proofs of theoretical properties of solutions and the optimization of numerical algorithms. The theoretical frameworks of conventional finite element methods and finite difference methods are well-established, enabling high-precision numerical approximation of elliptic problems and rendering them the mainstream tools for engineering computational analysis\cite{bre2008,Ciar2002}.
For second-order linear elliptic boundary value problems, conforming, nonconforming and mixed finite element methods have all attained mature development. Specifically targeting three-dimensional scenarios, Zhao Zhongjian and Chen Shaochun\cite{chenshao2020} proposed a novel low-order triangular prism mixed finite element scheme. From the perspective of elliptic operator theory, Dolbeault et al.\cite{Dol2015} conducted investigations centering on the symmetry of elliptic operators and the symmetric characteristics of associated solutions, and discussed the symmetric and monotonic behaviors of solutions to elliptic equations under specific assumptions. Most existing literature restricts analysis to particular symmetric constraints or special geometric domains. There lacks a systematic characterization of symmetry groups coupled with multiple factors including coefficient terms, source terms and boundary conditions for general second-order linear elliptic equations. Furthermore, a complete numerical reduction and domain decomposition framework built upon coupled multi-source symmetries has not yet been established. Considerable research gaps still exist regarding the integrated application of symmetry theories and high-performance finite element computation.\\
\indent Against the aforementioned research backdrop, this paper carries out systematic research on solution symmetry and domain reduction finite element methods for second-order linear elliptic Dirichlet boundary value problems. The chapter organization of this paper is arranged as follows. Chapter 2 formulates the standard mathematical model of second-order linear elliptic Dirichlet boundary value problems on bounded domains and introduces the corresponding well-posedness theorems. Chapter 3 analyzes the transformation properties of second-order linear elliptic differential operators under orthogonal coordinate transformations. On this basis, the symmetry groups of the second-order coefficient matrix, first-order coefficient vector, zero-order coefficient, interior source term and boundary source term involved in the boundary value problem are defined, and it is rigorously proven that the common symmetry group of all the above functional symmetry groups constitutes a subgroup of the solution symmetry group. Chapter 4 investigates the reduction mechanism of mirror symmetry elements on computational domains. It is verified that if mirror symmetry elements are contained in the common symmetry group, the global Dirichlet boundary value problem can be equivalently reduced to a boundary value problem on a subdomain, where the newly generated boundary of the subdomain satisfies the generalized homogeneous Neumann boundary condition. This further perfects the symmetry reduction theory for elliptic boundary value problems. Relying on the established theory, linear finite element methods are adopted to numerically solve the reduced subdomain model. Computational dimensions and mesh scales are cut down via symmetric domain reduction, thereby significantly lowering numerical computational overhead. Ultimately, one theoretical example together with three numerical experiments are presented to verify the correctness and efficiency of the proposed symmetry group theory and domain reduction algorithm in this work.

\section{The second-order linear elliptic Dirichlet boundary value problem}
Let $\Omega$ be a bounded Lipschitz domain in $\mathbb{R}^n$, and let $\partial\Omega$ denote the boundary of $\Omega$. Suppose $f({x})$ is a given $n$-variable function defined on $\Omega$, $h({x})$ is a given $n$-variable function defined on $\partial\Omega$, and $u({x})$ stands for the unknown $n$-variable function over $\Omega\cup\partial\Omega$.

Consider the following second-order linear differential boundary value problem:
\begin{subequations} \label{eq1}
\begin{numcases}{}
    \mathcal{L}(u(x)) = f(x) & $x\in \Omega$ \label{eq1a} \\
    u(x) = h(x) & $x\in \partial \Omega$ \label{eq1b}
\end{numcases}
\end{subequations}
where $\mathcal{L}$ denotes an $n$-dimensional second-order linear differential operator, which takes the form
\begin{equation}\label{eq2}
\mathcal{L}u = -\sum_{i,j=1}^{n} (a_{ij}(x)u_{x_i})_{x_j} + \sum_{i=1}^{n} b^{i}(x)u_{x_i} + c(x)u
\end{equation}
or
\begin{equation}\label{eq3}
 \quad \mathcal{L}u = -\sum_{i,j=1}^{n} a_{ij}(x)u_{x_ix_j} + \sum_{i=1}^{n} b_i(x)u_{x_i} + c(x)u
\end{equation}
The coefficient functions $\,a_{ij}(x)=a_{ji}(x),\,b_i(x),\,b^{i}(x)$, $c(x)\;(i,\,j = 1,\,\dots,\,n)$ are all given $n$-variable functions. From a physical perspective, $a_{ij}(x),\,b_i(x),\,b^i(x),\,c(x)\;(i,\,j = 1,\,\dots,\,n)$ represent the physical parameters of anisotropic media within the domain $\Omega$, $f(x)$ stands for the source term inside $\Omega$, $h(x)$ corresponds to the source term on the boundary $\partial\Omega$, and $u(x)$ denotes the unknown scalar physical field or a single component of a vector physical field defined over $\Omega$.

If $\mathcal{L}$ is defined by Eq.\,(\ref{eq2}), then $\mathcal{L}u = f$ is referred to as a second-order linear elliptic boundary value problem in divergence form. If $\mathcal{L}$ is given by Eq.\,(\ref{eq3}), then $\mathcal{L}u = f$ is called a second-order linear elliptic boundary value problem in general form. Equation\,(\ref{eq1b}) is named the Dirichlet boundary condition. When $h(x)\equiv0$, it reduces to the homogeneous Dirichlet boundary condition. Equation\,(\ref{eq1b}) can be replaced with Neumann, Robin or mixed boundary conditions. For brevity, this paper only investigates second-order linear elliptic boundary value problems equipped with Dirichlet boundary conditions.
\begin{definition}[Definition of Uniform Ellipticity]\label{defe1}
A second-order linear differential operator $\mathcal{L}$ is said to be uniformly elliptic if there exists a constant $\alpha> 0$ such that
\begin{equation*}
\quad \sum_{i,j=1}^{n} a_{ij}(x)\xi_i\xi_j \ge \alpha |\xi|^2
\end{equation*}
holds for almost every $x\in\Omega$ and all $\xi \in \mathbb{R}^n$.
\end{definition}

Let $A(x)=(a_{ij}(x))_{n\times n}$ denote an $n$-order matrix function defined on the bounded domain $\Omega$. If $\mathcal{L}$ is uniformly elliptic over the bounded domain $\Omega$, then the matrix function $A(x)$ is uniformly positive definite on $\Omega$. Let $\hat{b}(x)=(b^1(x),\,b^2(x),\,\cdots,b^{n}(x))^{T}$ and $b(x)=(b_1(x),\,b_2(x),\,\cdots,\,b_{n}(x))^{T}$ be two $n$-dimensional column vector functions. The gradient of $u$ is defined as $\nabla u=\big(\frac{\partial u}{ \partial x_1},\frac{\partial u}{\partial x_2},\cdots,\frac{\partial u}{\partial x_{n}}\big)^{T}$, which is also an $n$-dimensional column vector function.

When $\mathcal{L}$ is given in divergence form by \eqref{eq2}, the second-order linear uniformly elliptic differential operator $\mathcal{L}$ can be rewritten as
\begin{equation}\label{diver}
\mathcal{L}u=-\nabla\cdot(A(x)\nabla u)+\hat{b}(x)\cdot \nabla u+c(x)u(x),\quad x\in{\Omega}
\end{equation}
where $\nabla\cdot$ stands for the $n$-dimensional divergence operator.

When $\mathcal{L}$ is expressed in the general form \eqref{eq3}, the second-order linear uniformly elliptic differential operator $\mathcal{L}$ can be rewritten as
\begin{equation}\label{diver2}
\mathcal{L}u=-A(x):D^2u+b(x)\cdot\nabla u+c(x)u(x),\quad x\in{\Omega}
\end{equation}
where $D^2u=(u_{x_{i}x_{j}})_{n\times n}$ denotes the Hessian matrix of the $n$-variable function $u$, and the colon operator $:$ represents the Frobenius inner product for two matrices of identical dimension.

If the second-order coefficients satisfy $a_{ij}(x)\in C^1(\Omega)$ for all $i,j = 1, \dots, n$, then the second-order linear elliptic differential operator in divergence form \eqref{diver} can be transformed into the general-form second-order linear elliptic differential operator \eqref{diver2}. In fact, the following identity holds:
\begin{equation}\label{firstordersd}
b(x)=\hat{b}(x)-A(x)\nabla_{x}
\end{equation}
where $\nabla_{x}=\big(\frac{\partial}{\partial x_1},\,\frac{\partial}{\partial x_2},\,\cdots,\,\frac{\partial}{\partial x_n}\big)^{T}$.

Consider the second-order linear elliptic Dirichlet boundary value problem:
\begin{subequations} \label{wellpose}
\begin{numcases}{}
    \mathcal{L}(u(x))-\mu u(x) = f(x) & $x\in \Omega$ \label{wellpose1a} \\
    u(x) = h(x) & $x\in \partial \Omega$ \label{wellpose1b}
\end{numcases}
\end{subequations}
where the constant $\mu\in{\mathbb{C}}$ is known.

The corresponding eigenvalue problem for equation \eqref{wellpose} reads
\begin{subequations} \label{eigen}
\begin{numcases}{}
   \mbox{Seek}\,\lambda\in{\mathbb{C}},~u(x)\not\equiv 0,~\mbox{such that}\nonumber\\
    \mathcal{L}(u(x))=\lambda u(x)  & $x\in \Omega$ \label{eigen1a} \\
    u(x) = 0 & $x\in \partial \Omega$ \label{eigen1b}
\end{numcases}
\end{subequations}
where $\lambda$ and $u(x)$ are referred to as the eigenvalue and eigenfunction of the eigenvalue problem \eqref{eigen} for the second-order linear elliptic differential operator, respectively. The eigenvalue problem \eqref{eigen} corresponds to an eigenvalue problem of a compact operator. From the spectral theory of compact operators \cite{Conway}, the spectrum of \eqref{eigen} consists solely of the point spectrum with neither continuous spectrum nor residual spectrum. This point spectrum is composed of countable complex numbers, and the unique accumulation point is infinity.

Based on the well-posedness theory of second-order linear elliptic differential equations \cite{Evans}, the following theorem holds.

\begin{theorem}[Fredholm Alternative]\label{theorem1}
If $\mu$ is not an eigenvalue of the eigenvalue problem \eqref{eigen}, then for any $f\in H^{-1}(\Omega)$ and $h\in H^{\frac{1}{2}}(\partial\Omega)$, equation \eqref{wellpose} admits a unique weak solution $u\in H^1(\Omega)$ satisfying the following stability estimate
\[
\|u\|_{H^1(\Omega)}\leq C\big(\|f\|_{H^{-1}(\Omega)}+\|h\|_{H^{\frac{1}{2}}(\partial\Omega)}\big)
\]
where $C$ is a positive constant. If $\mu$ is an eigenvalue of the eigenvalue problem \eqref{eigen}, the boundary value problem \eqref{wellpose} either has no solution or possesses non-unique solutions.
\end{theorem}

Comparing \eqref{eq1} with \eqref{wellpose}, one can see that the boundary value problem \eqref{eq1} is a special case of equation \eqref{wellpose} by setting $\mu=0$. Combined with Theorem \ref{theorem1}, we conclude that the second-order linear elliptic Dirichlet boundary value problem \eqref{eq1} is well-posed provided that $0$ is not an eigenvalue of the eigenvalue problem \eqref{eigen}.

The second-order linear elliptic boundary value problem \eqref{eq1} has extremely wide applications in engineering and physics, and it is particularly valuable for problems such as electromagnetic fields, electromagnetic scattering, and acoustic scattering. For instance:
\begin{itemize}
\item \textbf{Electrostatic field problem}: For $n=3$, set $A(x) = \epsilon(x)I_{3}$, where $\epsilon(x)$ denotes the permittivity of isotropic media, $I_{3}$ stands for the $3\times 3$ identity matrix, $b_i(x) = c(x) = 0$, $h(x)=0$, and $f(x)=\rho(x)$ is the given charge density distribution. The governing equation becomes $-\nabla\cdot(\epsilon(x)\nabla u) = \rho(x)$. The physical unknown $u$ represents the electric potential, and Eq. \eqref{eq1} models the potential distribution inside isotropic media when the boundary $\partial\Omega$ of the domain is grounded. In principle, the potential distribution for electrostatic problems on bounded regular three-dimensional domains can be solved analytically via the Green��s function method.

\item \textbf{Static magnetic field problem}: For $n=3$, consider the scalarized formulation of the magnetic vector potential with
$A(x) = \mu^{-1}(x) I_{3}$, where $\mu(x)$ is the magnetic permeability of the medium, $b_i = c = 0$, $h(x)=0$, and $f(x)=J(x)$ denotes the prescribed current source term. The resulting equation reads
\[
-\nabla \cdot \big(\mu^{-1}(x)\nabla u\big) = J(x),
\]
where the unknown $u$ corresponds to a single component of the magnetic vector potential. This equation characterizes the spatial distribution of magnetic fields within isotropic media under specified current excitations and boundary conditions.

\item \textbf{Time-harmonic scalar electromagnetic field problem}: For $n=2$ or $n=3$, consider lossless media without first-order derivative terms. Set
$A(x) = I_n$, $c(x) = -k^2$ with $k$ being the wave number, and $h(x)=0$. The equation reduces to
\[
-\Delta u - k^2 u = f,
\]
where the unknown $u$ denotes a scalar component of the electric or magnetic field. This equation governs the propagation and scattering of electromagnetic waves in the frequency domain, and is widely adopted in waveguide modal analysis and electromagnetic scattering simulations.

\item \textbf{Acoustic radiation and scattering problem}: For $n=2$ or $n=3$, consider homogeneous fluid media with
$A(x) = I_n$, $c(x) = -k^2$, $b_i = 0$, $h(x)=0$, and $f(x)$ representing the acoustic source term. The governing equation takes the same form as that for time-harmonic scalar electromagnetic fields.
The unknown function $u$ stands for the acoustic pressure or velocity potential, and $k = \omega/c_0$ is the acoustic wave number. This model describes the propagation, radiation and scattering of steady-state acoustic waves in the frequency domain, with extensive engineering applications including loudspeaker design, indoor sound field analysis, and underwater acoustic detection.
\end{itemize}

All the above physical problems can be uniformly formulated as second-order elliptic partial differential equations with different coefficient configurations, which demonstrates the unified modeling capability of such equations in electromagnetics, acoustics and continuum mechanics.

\section{The symmetry of the solution to second-order elliptic Dirichlet boundary value problems}
Based on fundamental results from classical group theory, this section explores the symmetry of solutions for second-order linear elliptic Dirichlet boundary value problems. This symmetry stems from the symmetric characteristics of the coefficient function, volume source, and boundary source in Eq.\eqref{eq1}.

Let $\Omega$ be a nonempty bounded domain in the Euclidean space $\mathbb{R}^{n}$, whose boundary $\partial\Omega$ is Lipschitz continuous. We establish an $n$-dimensional Cartesian coordinate system with the geometric center $O$ of $\Omega$ taken as the coordinate origin, and denote its standard orthonormal basis by $\{\mathbf{e}_{1}, \mathbf{e}_{2}, \dots, \mathbf{e}_{n}\}$. For any point $x\in\Omega$, it admits a unique representation $x = \sum_{i=1}^{n} x_i \mathbf{e}_i$, where $(x_1, x_2, \dots, x_n)\in\mathbb{R}^n$ is referred to as the Cartesian coordinates of the point $x$ under this coordinate system. The classical Euclidean norm of a vector $x$ in $\mathbb{R}^{n}$ is defined as $\|x\|=\big(\sum_{k=1}^{n}x_{k}^2\big)^{1/2}$. To simplify notation, we frequently identify the point $x$ with its coordinate vector $(x_1, x_2, \dots, x_n)$ when no ambiguity arises.

\subsection{Symmetry Groups of Domains and Functions}
\indent Let $S(\Omega)$ denote the set consisting of all bijections $f:\Omega\rightarrow\Omega$ defined on $\Omega$. From fundamental classical algebraic theory \cite{sitikejin,artin2014}, all bijections in $S(\Omega)$ form a group under the composition operation $\circ$ of mappings. The group $(S(\Omega),\circ)$ is referred to as the symmetry group on the domain $\Omega$, which is commonly abbreviated as $S(\Omega)$.

\begin{definition}[Isometry Group]
Let $\Omega \subset \mathbb{R}^n$ be a nonempty bounded domain. The set
\begin{equation*}
\operatorname{Sym}(\Omega) = \Big\{ g \in S(\Omega) : \| g(x) - g(y) \| = \| x - y \|,\ \forall\, x, y \in \Omega \Big\}
\end{equation*}
is called the isometry group of the $n$-dimensional bounded domain $\Omega$.
\end{definition}

\begin{lemma}\label{lemma1}
Let $\Omega\subset\mathbb{R}^n$ be a nonempty bounded domain. Then every isometry $g\in\operatorname{Sym}(\Omega)$ admits a unique extension to an isometry $\tilde{g}$ on $\mathbb{R}^n$ of the form
\[
\tilde{g}(x)=A_{\tilde{g}}x+b_{\tilde{g}},
\]
where $A_{\tilde{g}}$ is an $n$-order orthogonal matrix uniquely determined by the extension $\tilde{g}$, $b_{\tilde{g}}$ is an $n$-dimensional real column vector uniquely determined by $\tilde{g}$, and the restriction satisfies $\tilde{g}|_{\Omega}=g$.
\end{lemma}

For the proof of the theory related to Lemma\,\ref{lemma1}, see reference \cite{bur2001}. According to the conclusion of Lemma\,\ref{lemma1}, the following theorem can be obtained naturally.
\begin{theorem}\label{thegroyp}
Let $\m{\Omega\subset\mathbb{R}^n}$ be a nonempty bounded domain. Then $\,\operatorname{Sym}(\Omega)\,$ is isomorphic to a subgroup of the $\m{n}$th-order orthogonal matrix group $\m{O_n}$.
\end{theorem}
According to Theorem\,\ref{thegroyp}, to simplify notation, we regard $\m{\operatorname{Sym}(\Omega)}$ as a subgroup of the $\m{n}$th-order orthogonal matrix group $\m{O_n}$, and directly denote this subgroup by $\m{\operatorname{Sym}(\Omega)}$.

\begin{definition}[Symmetry group of a function]
Let $\m{v(x)}$ be an $\,{n}$-variable function defined on a nonempty bounded domain $\m{\Omega}$. The set
\begin{equation}
\operatorname{Sym}(v)=\{Q\in \operatorname{Sym}(\Omega):\ v(Qx)=v(x),\ \forall x\in\Omega\}
\end{equation}
is called the symmetry group of the $\,n$-variable function $\m{v}$.
\end{definition}
By the subgroup criterion, it is easy to prove that $\m{\operatorname{Sym}(v)$ is a subgroup of $\,\operatorname{Sym}(\Omega)}.$

\subsection{Coordinate Transformation and Function Transformation}\label{section3-2}
\indent Introduce the variable substitution $\m{x = Qy}$ and the function transformation $\,{u(x) = u(Qy) = g(y)}$, where $\m{Q=(q_{ij})}$ is an $\m{n}$th-order orthogonal matrix. The main goal of this subsection is to derive the partial differential equation for the unknown function $\m{g(y)}$ from the second-order linear elliptic differential boundary value problem\,(\ref{eq1})\,with unknown function $\m{u(x)}$, based on the general form (\ref{eq3}) of the operator $\m{\mathcal{L}}$.\\
\indent Let $\,x,y\in\mathbb{R}^n$, and suppose that there exists an $\m{n}$th-order orthogonal matrix $\m{Q}$ such that $\,x = Qy$.
Then $\m{y = Q^T x}$ holds. At this time, the components of $\m{y}$ can be expressed as $\,y_k=\sum_{j=1}^{n}q_{jk}x_j$. Differentiating this with respect to $\m{x_i}$ gives $\frac{\partial y_k}{\partial x_i} = q_{ik}$. On this basis, by the multivariable chain rule, the transformation formula for the first derivative is
\begin{equation}\label{eq4}
	\frac{\partial u}{\partial x_i} = \sum_{k=1}^n q_{ik}\frac{\partial g}{\partial y_k}
\end{equation}
and the transformation formula for the second derivative is
\begin{equation}\label{eq5}
	\frac{\partial^2 u}{\partial x_i \partial x_j} = \sum_{k,l=1}^n q_{ik}q_{jl}\frac{\partial^2 g}{\partial y_k \partial y_l}
\end{equation}
Substituting\,(\ref{eq4})-(\ref{eq5})\,and $\m{x=Qy}$ into\,(\ref{eq3}) gives
\begin{equation}\label{eq6}
-\sum_{i,j=1}^{n}a_{ij}(Qy)\Bigg(\sum_{k,l=1}^{n} q_{ik}q_{jl}\frac{\partial^2 g}{\partial y_k \partial y_l}\Bigg) + \sum_{i=1}^{n}b_i(Qy) \sum_{k=1}^{n} q_{ik}\frac{\partial g}{\partial y_k} + c(Qy)g(y)=f(Qy)
\end{equation}
Exchanging the order of summation in equation\,(\ref{eq6}) gives
\begin{equation}\label{eq7}
-\sum_{k,l=1}^{n} \left[ \sum_{i,j=1}^{n} a_{ij}(Qy)q_{ik}q_{jl} \right] \frac{\partial^2 g}{\partial y_k \partial y_l} + \sum_{k=1}^{n} \left[ \sum_{i=1}^{n} b_i(Qy)q_{ik} \right] \frac{\partial g}{\partial y_k} + c(Qy)g(y)=f(Qy)
\end{equation}
In equation\,(\ref{eq7}), define the coefficient functions under the new variable $\m{y}$ by
\begin{equation}\label{effci1}
	\begin{aligned}
		\tilde{a}_{kl}(y) = \sum_{i,j=1}^{n} a_{ij}(Qy)q_{ik}q_{jl},\quad \tilde{b}_k(y) = \sum_{i=1}^{n} b_i(Qy)q_{ik},\quad \tilde{c}(y)=c(Qy), \quad\tilde{f}(y)=f(Qy)		
	\end{aligned}
\end{equation}
Using the rules of matrix multiplication in linear algebra, the first two expressions in\,\eqref{effci1}\,can be further arranged into the following matrix form:
\begin{equation}\label{effci2}
	\tilde{A}(y) = Q^T A(Qy)Q, \quad \tilde{b}(y)=Q^T b(Qy)
\end{equation}
Here $\m{\tilde{A}(y)=(\tilde{a}_{kl}(y))}$ is an $\m{n}$th-order square matrix function, and $\m{\tilde{b}(y)=(\tilde{b}_{k}(y))}$ is an $\m{n}$-dimensional column-vector function.

Substituting\,\eqref{effci1}\,into equation\,\eqref{eq7} gives
\begin{equation}\label{eq8}
 -\sum_{k,l=1}^{n}\tilde{a}_{kl}(y) \frac{\partial^2 g}{\partial y_k \partial y_l} + \sum_{k=1}^{n} \tilde{b}_k(y) \frac{\partial g}{\partial y_k} + \tilde{c}(y)g(y)=\tilde{f}(y)
\end{equation}

Let the operator under the new coordinate system be
\begin{equation}
	\tilde{\mathcal{L}} = -\sum_{i,j=1}^{n} \tilde{a}_{ij}(y) \frac{\partial^2}{\partial y_i \partial y_j} + \sum_{i=1}^{n} \tilde{b}_i(y) \frac{\partial}{\partial y_i} + \tilde{c}(y)
\end{equation}

Since $\m{\mathcal{L}}$ is uniformly elliptic, $\m{A(x)}$ is a uniformly positive definite matrix function, and hence $\m{A(Qy)}$ is also a uniformly positive definite matrix function. According to\,(\ref{effci2}), $\m{\tilde{A}(y)}$ is also a uniformly positive definite matrix function, and consequently $\m{\tilde{\mathcal{L}}}$ is also uniformly elliptic. Under the variable substitution $\m{x = Qy}$ and the function transformation $\m{u(x) = g(y)}$, the second-order elliptic boundary value problem\,(\ref{eq1})\,can be transformed into
\begin{subequations} \label{eq9}
\begin{numcases}{}
    \mathcal{\tilde L}(g(y)) = \tilde{f}(y) & $y\in \Omega$ \label{eq9a} \\
    g(y) = \tilde{h}(y) & $y\in \partial \Omega$ \label{eq9b}
\end{numcases}
\end{subequations}
where $\tilde{h}(y)=h(Qy)$.

\begin{remark}
According to the above argument, the following conclusion can be obtained directly: the classical $\m{n}$-dimensional\,Laplace\,operator remains invariant under orthogonal transformations, that is, it has orthogonal invariance.
\end{remark}

\subsection{Symmetry Groups of the Second-Order Linear Elliptic Operator and Sources}
Using transformation groups to study the invariance of differential equations is a basic idea in symmetry analysis of second-order linear partial differential equations\cite{Olver1993}.
This section discusses the invariance of equation\,(\ref{eq1})\,under orthogonal transformations, and it is necessary to impose corresponding symmetry conditions on the coefficient functions appearing in the equation. According to the transformation relations\,\eqref{effci1}-(\ref{effci2}), the symmetry groups of the coefficients of the second-order linear elliptic differential operator $\m{\mathcal{L}}$ in its general form\,\eqref{eq3}\,are readily obtained.

The symmetry group of the second-order coefficient matrix function $\,A(x)\,$ is
\begin{equation}\label{coffer1}
\operatorname{Sym}(A)
=
\{Q\in \operatorname{Sym}(\Omega):\ Q^TA(Qx)Q=A(x),\ \forall x\in\Omega\}
\end{equation}

The symmetry group of the first-order coefficient column-vector function $\,b(x)\,$ is
\begin{equation}\label{coffer2}
\operatorname{Sym}(b)
=
\{Q\in \operatorname{Sym}(\Omega):\ Q^Tb(Qx)=b(x),\ \forall x\in\Omega\}
\end{equation}

The symmetry group of the zeroth-order coefficient function $\,c(x)\,$ is
\begin{equation}\label{coffer3}
\operatorname{Sym}(c)=\{Q\in \operatorname{Sym}(\Omega):\ c(Qx)=c(x),\ \forall x\in\Omega\}
\end{equation}

\begin{definition}[Symmetry group of the uniformly elliptic differential operator $\m{\mathcal{L}}$ in general form]\label{defineL1}
For a given uniformly elliptic linear differential operator $\,\mathcal{L}\,$ in general form, define
\begin{equation*}
\operatorname{Sym}(\mathcal{L})
=
\operatorname{Sym}(A)
\cap
\operatorname{Sym}(b)
\cap
\operatorname{Sym}(c)
\end{equation*}
as the symmetry group of the uniformly elliptic linear differential operator $\,\m{\mathcal{L}}\,$ in general form.
\end{definition}
\indent By the subgroup criterion, it is easy to prove that $\,\operatorname{Sym}(A),\,\operatorname{Sym}(b),\,\operatorname{Sym}(c)\,$ are all subgroups of $\,\operatorname{Sym}(\Omega)\,$; hence the symmetry group $\,\operatorname{Sym}(\mathcal{L})\,$ of the uniformly elliptic differential operator in general form is also a subgroup of $\,\operatorname{Sym}(\Omega)\,$.

Next consider the symmetry group of the uniformly elliptic differential operator $\m{\mathcal{L}}$ in divergence form. The uniformly elliptic differential operator $\m{\mathcal{L}}$ differs between general form and divergence form only in the coefficient function of the first-order term. Consider the orthogonal transformation $\,x=Qy$, where $\,Q\in{\operatorname{Sym}(A)}\cap\operatorname{Sym}(b)$. Then $\,\nabla_{x}=Q\nabla_{y}$, where $\,\nabla_{y}=(\frac{\partial~}{\partial y_1},\,\frac{\partial~}{\partial y_2},\,\cdots,\,\frac{\partial~}{\partial y_n})^{T}$. Setting $\,x=Qy$ in\,\eqref{firstordersd}\,gives
\begin{equation}\label{firstordery}
b(Qy)=\hat{b}(Qy)-A(Qy)Q\nabla_{y},\ \forall y\in\Omega
\end{equation}
Since $\,Q\in{\operatorname{Sym}(A)}$, we have
\begin{equation}\label{deduce1}
 Q^{T}A(Qy)Q=A(y),\ \forall y\in\Omega
\end{equation}
Since $\,Q\in{\operatorname{Sym}(b)}$, we have
\begin{equation}\label{deduce2}
Q^{T}b(Qy)=b(y),\ \forall y\in\Omega
\end{equation}
Substituting\,\eqref{deduce1}-\eqref{deduce2}\,into\,\eqref{firstordery} yields
\begin{equation}\label{firstor}
b(y)=Q^{T}\hat{b}(Qy)-A(y)\nabla_{y},\ \forall y\in\Omega
\end{equation}
Replacing $\,y\,$ in\,\eqref{firstor}\,by $\,x\,$ and then comparing this expression with\,\eqref{firstordersd}\,one obtains that the symmetry group of the first-order coefficient column-vector function $\,\hat{b}(x)\,$ in divergence form is
\begin{equation}\label{divcoffer2}
\operatorname{Sym}(\hat{b})
=
\{Q\in \operatorname{Sym}(\Omega):\ Q^T\hat{b}(Qx)=b(x),\ \forall x\in\Omega\}
\end{equation}
\begin{definition}[Symmetry group of the uniformly elliptic differential operator $\m{\mathcal{L}}$ in divergence form]\label{defineL2}
For a given uniformly elliptic linear differential operator $\mathcal{L}$ in divergence form, define
\begin{equation*}
\operatorname{Sym}(\mathcal{L})
=
\operatorname{Sym}(A)
\cap
\operatorname{Sym}(\hat{b})
\cap
\operatorname{Sym}(c)
\end{equation*}
as the symmetry group of the uniformly elliptic linear differential operator $\m{\mathcal{L}}$ in divergence form.
\end{definition}

According to Definition \ref{defineL1} and Definition \ref{defineL2}, the definitions of the symmetry group $\m{\operatorname{Sym}{(\mathcal{L})}}$ of the uniformly elliptic linear differential operator in divergence form and in general form are the same.

For the source $\m{f}$ inside the domain $\m{\Omega}$ and the source $\,{h}$ on the boundary $\,{\partial\Omega}$, their symmetry groups are defined respectively by
\begin{gather*}
\operatorname{Sym}(f)=\{Q\in \operatorname{Sym}(\Omega):\ f(Qx)=f(x),\ \forall x\in\Omega\}\\
\operatorname{Sym}(h)=\{Q\in \operatorname{Sym}(\Omega):\ h(Qx)=h(x),\ \forall x\in\partial\Omega\}
\end{gather*}
\begin{definition}[Symmetry group of the total source]\label{definesouce}
For the source $\m{f}$ inside the domain $\m{\Omega}$ and the source $\,{h}$ on the boundary $\m{\partial\Omega}$, define
\begin{equation*}
\operatorname{Sym}(f,h)
=
\operatorname{Sym}(f)
\cap
\operatorname{Sym}(h)
\end{equation*}
as the symmetry group of the total source.
\end{definition}

\indent By the subgroup criterion, it is easy to prove that $\,\operatorname{Sym}(f)\,$ and $\,\operatorname{Sym}(h)\,$ are both subgroups of $\,\operatorname{Sym}(\Omega)\,$; hence the symmetry group $\,\operatorname{Sym}(f,h)\,$ of the total source is also a subgroup of $\,\operatorname{Sym}(\Omega)\,$.

\subsection{Solution Symmetry Theorem}
For the well-posed second-order linear elliptic\,Dirichlet\,boundary value problem\,\eqref{eq1}, if an orthogonal transformation simultaneously leaves the domain, coefficient functions, source term, and boundary condition invariant,
then the transformed function still satisfies the same boundary value problem. By uniqueness of the solution, it follows that the solution inherits the common symmetry group of the equation coefficients.
\begin{theorem}[Solution symmetry theorem]\label{thm:symmetry_solution}
Assume that the second-order linear elliptic\,Dirichlet\,boundary value problem\,\eqref{eq1}\,is well-posed, and let $G=\operatorname{Sym}(\mathcal{L})\cap \operatorname{Sym}(f,h)$. Then
$G\subseteq \operatorname{Sym}(u)$, that is, the solution $\,u(x)\,$ inherits at least the common symmetry of the equation data.
\end{theorem}
\begin{proof}
Let $\m{u}$ be the unique solution of the following well-posed second-order linear elliptic\,Dirichlet\,boundary value problem
\begin{subequations}\label{pwhole}
\begin{numcases}{}
   \mathcal{L}u=-A(x):D^2u(x)+b(x)\cdot\nabla u(x)+c(x)u(x) = f(x) & $x\in \Omega$ \label{pwholea} \\
    u(x) = h(x) & $x\in \partial \Omega$ \label{pwholeb}
\end{numcases}
\end{subequations}
Take any $\,Q\in G$, and define the transformed function $\,u_Q(x)=u(Qx)$. Since $\,Q\in\operatorname{Sym}(\Omega)$, we have $\,Q\Omega=\Omega$; hence $\,u_Q\,$ is still defined on the domain $\m{\Omega}$. According to Section\,\ref{section3-2}, $u_{Q}\,$ satisfies the following second-order linear elliptic\,Dirichlet\,boundary value problem:
\begin{subequations}\label{ppwhole}
\begin{numcases}{}
   \begin{aligned}
   &-Q^{T}A(Qx)Q:D^2u_{Q}(x)
   +Q^{T}b(Qx)\cdot\nabla u_{Q}(x)  \\
   &\quad +c(Qx)u_{Q}(x)=f(Qx)
   \end{aligned}
   &$x\in \Omega$ \label{ppwholea} \\[0.5em]
   u_{Q}(x)=h(Qx) &$x\in \partial \Omega$ \label{ppwholeb}
\end{numcases}
\end{subequations}
Since $\,Q\in{G}$, we have
\begin{gather}
  Q^{T}A(Qx)Q=A(x),~~Q^{T}b(Qx)=b(x),~~c(Qx)=c(x),~~\forall x\in{\Omega} \label{iterative1}\\
  f(Qx)=f(x),~~h(Qx)=h(x),~~\forall x\in{\Omega}\label{iterative2}
\end{gather}
Substituting\,\eqref{iterative1}-\eqref{iterative2}\,into\,\eqref{ppwhole}, $\m{u_{Q}}$ satisfies the following second-order linear elliptic\,Dirichlet\,boundary value problem:
\begin{subequations}\label{swhole}
\begin{numcases}{}
   \mathcal{L}u_{Q}=-A(x):D^2u_{Q}(x)+b(x)\cdot\nabla u_{Q}(x)+c(x)u_{Q}(x) = f(x) & $x\in \Omega$ \label{swholea} \\
    u_{Q}(x) = h(x) & $x\in \partial \Omega$ \label{swholeb}
\end{numcases}
\end{subequations}
Comparing\,\eqref{pwhole}\,and\,\eqref{swhole}, we see that $u\,$ and $\,u_Q\,$ satisfy the same equation and boundary condition. Since the boundary value problem\,\eqref{pwhole}\,is well-posed, we have $\,u(Qx)=u_Q(x)=u(x),\,\forall x\in\Omega$, and therefore $\,Q\in\operatorname{Sym}(u)$.
Since $\,Q\in{G}\,$ is arbitrary, $\,G\subseteq \operatorname{Sym}(u)$. This completes the proof.
\end{proof}

\indent Although the above proof is for the general form of the second-order linear elliptic\,Dirichlet\,boundary value problem, the same proof idea also applies to the divergence form of the second-order linear elliptic\,Dirichlet\,boundary value problem, and the details are omitted here.

\begin{remark}\label{remake1}
When the isometry group $\m{\text{Sym}(\Omega)}$ of the $\m{n}$-dimensional bounded domain $\m{\Omega}$ degenerates into the trivial group, the symmetry group $\m{\text{Sym}(u)}$ of the solution of equation\,(\ref{eq1})\,is also the trivial group; in this case the solution $\m{u(x)}$ has no symmetry.
\end{remark}
\begin{remark}\label{remake2}
If the common symmetry group $\,\text{Sym}(\mathcal{L})\cap\text{Sym}(f,g)\,$ of equation\,\eqref{eq1}\,is the isometry group $\m{\text{Sym}(\Omega)}$ of the $\m{n}$-dimensional bounded domain $\m{\Omega}$, then the solution symmetry group of equation\,(\ref{eq1})\,is $\m{\text{Sym}(u)}=\text{Sym}(\Omega)$; in this case the solution $\m{u(x)}$ has full symmetry.
\end{remark}

\section{Finite Element Method Based on Domain-Reduction}
\indent When the common symmetry group $\,\operatorname{Sym}(\mathcal{L})\cap\text{Sym}(f,h)\,$ of the second-order linear elliptic\,Dirichlet\,boundary value problem\,\eqref{eq1}\,contains several reflection-symmetry elements, the original boundary value problem on the whole domain $\,\Omega\,$ can be reduced to the corresponding problem on a certain subdomain, and a generalized homogeneous\,Neumann\,boundary condition appears on the boundary of that subdomain. For the case where $\m{\Omega\subset\mathbb{R}^{n}\,(n=2,3)}$ is a bounded domain, this section uses a linear finite element method to numerically solve the reduced boundary value problem on the subdomain, thereby realizing domain reduction and significantly reducing the computational effort.
\subsection{Domain-Reduced Second-Order Linear Elliptic Equation}
\indent Compared with the general form, the divergence form of a second-order linear elliptic boundary value problem has a more concise variational structure. Therefore, this section adopts the boundary value problem in divergence form as the mathematical model, stated specifically as follows:
\begin{subequations}\label{domain}
\begin{numcases}{}
    \mathcal{L}u=-\nabla\cdot(A(x)\nabla u)+\hat{b}(x)\cdot \nabla u+c(x)u(x) = f(x) & $x\in \Omega$ \label{domaina} \\
    u(x) = h(x) & $x\in \partial \Omega$ \label{domainb}
\end{numcases}
\end{subequations}
where $\m{A(x)}$ is uniformly positive definite in $\m{\Omega}$. Assume that the boundary value problem\,\eqref{domain}\,satisfies the conditions of Theorem\,\ref{theorem1}. Then the boundary value problem\,\eqref{domain}\,is well-posed.

\indent Suppose that the common symmetry group $\,\operatorname{Sym}(\mathcal{L})\cap\text{Sym}(f,h)$ of the second-order linear elliptic boundary value problem\,\eqref{domain}\,contains $\,m\,$ reflection-symmetry elements, and denote the corresponding mirrors by $\,\{\Pi_{1},\,\Pi_{2},\,\cdots,\,\Pi_{m}\}$. By Theorem\,\ref{thm:symmetry_solution}, the solution symmetry group $\,\text{Sym}(u)\,$ also contains these $\,m\,$ reflection-symmetry elements. Clearly, these $\m{m}$ mirrors $\,\Pi_1, \Pi_2, \cdots, \Pi_m\,$ all pass through the geometric center $\,O\,$ of the domain $\,\Omega\,$. Let the unit normal vector of the mirror $\m{\Pi_k}$ be the unit column vector $\,\boldsymbol{n}_k$ in $\,\mathbb{R}^{n}\,$. Then the reflection matrix corresponding to the mirror-reflection transformation on the mirror $\,\Pi_{k}\,$ is $\,Q_k = I - 2\boldsymbol{n}_k \boldsymbol{n}_k^{T}$. Consequently, for any $\,k\in\{1,2,\cdots,m\}$, one has $\,Q_{k}\in \operatorname{Sym}(\mathcal{L})\cap\text{Sym}(f,h)$.

Suppose that these $\m{m}$ mirrors $\,\{\Pi_{1},\,\Pi_{2},\,\cdots,\,\Pi_{m}\}\,$ divide the $\m{n}$-dimensional bounded domain $\m{\Omega}$ into $\m{s}$ mutually disjoint congruent subdomains, denoted by $\,\Omega_1,\,\Omega_2,\,\dots,\,\Omega_{s}$, respectively. Then the second-order linear elliptic boundary value problem\,\eqref{domain}\,on the whole domain $\m{\Omega}$ can be reduced to the corresponding problem on any subdomain $\,\Omega_{j},\,j\in\{1,2,\cdots,s\}$. Without loss of generality, take $\,j=1$ and consider only the second-order linear elliptic differential boundary value problem on $\,\Omega_{1}\,$. The boundary $\,\partial\Omega_{1}\,$ of the subdomain $\,\Omega_{1}\,$ consists of the following two types: the first is a portion $\,\Gamma_{0}=\partial\Omega\cap\partial\Omega_{1}$ of the boundary $\,\partial\Omega\,$ of the whole domain $\m{\Omega}$; the second is the internal boundary $\,\Gamma_{k}=\partial\Omega_{1}\cap\Pi_{k}$ inside the domain $\,\Omega\,$, where $\m{k}$ belongs to the index set $\,\mathcal{I} = \{ k \in \{1,\dots,m\} \mid \partial\Omega_1 \cap \Pi_k \neq \varnothing\}$.

We next give the boundary condition satisfied by the solution $\,u(x)\,$ of equation\,\eqref{domain}\,on $\m{\Gamma_{k}\,(k\in{\mathcal{I}})}$. Since the symmetry group $\,\text{Sym}(u)\,$ contains at least the above $\,m\,$ reflection-symmetry elements, one has $u(Q_{k}x)=u(x),\,\forall x\in{\Omega}$. In particular, if $\,x\in{\Gamma_{k}}\,(k\in{\mathcal{I}})$ and $\,x\pm t\boldsymbol{n}_k\in{\Omega}$ for some $\,t\in{\mathbb{R}}\,$, then
\begin{equation}\label{symmetrypo}
 u(x+t\boldsymbol{n}_k)=u(x-t\boldsymbol{n}_k)
\end{equation}
By the definition of the normal derivative, one has
\begin{equation}\label{differntial}
\nabla u\cdot \boldsymbol{n}_k=\frac{\partial u}{\partial \boldsymbol{n}_k}=\lim\limits_{t \to 0^{+}}\frac{u(x+t\boldsymbol{n}_k)-u(x)}{t}=\lim\limits_{t \to 0^{+}}\frac{u(x)-u(x-t\boldsymbol{n}_k)}{t}
\end{equation}
By\,\eqref{symmetrypo}-\eqref{differntial} and the arithmetic properties of limits, it follows that
\begin{equation}\label{differntia2}
\nabla u\cdot \boldsymbol{n}_k=\frac{\partial u}{\partial \boldsymbol{n}_k}=\lim\limits_{t \to 0^{+}}\frac{u(x+t\boldsymbol{n}_k)-u(x-t\boldsymbol{n}_k)}{2t}=0
\end{equation}
From\,\eqref{differntia2}\,we know that $\nabla u\in{\Gamma_{k}}$; hence
\begin{equation}\label{gradus}
  Q_{k}\nabla u(x)=\nabla u(x),\qquad \forall\,x\in{\Gamma_{k}}
\end{equation}

Since $\,Q_{k}\in{\text{Sym}(A)}$, one has $\,A(x)=Q_{k}^{T}A(Q_{k}x)Q_{k},\,\forall x\in{\Omega}$.
When $\,x\in \Gamma_k\,$, one has $\,Q_{k}x=x$. Therefore,
\begin{equation}\label{symmetryyq}
  Q_{k}A(x)=A(x)Q_{k},\qquad \forall\,x\in{\Gamma_{k}}
\end{equation}
According to\,\eqref{gradus}-\eqref{symmetryyq}, when $\,x\in{\Gamma_{k}}$, one has
\begin{equation}\label{fluxb}
  Q_{k}(A(x)\nabla u(x))=(Q_{k}A(x))\nabla u(x)
=(A(x)Q_{k})\nabla u(x)=A(x)(Q_{k}\nabla u(x))
=A(x)\nabla u(x)
\end{equation}
By\,\eqref{fluxb}\,for any $\,x\in{\Gamma_{k}}$, the vector field $\,A(x)\nabla u(x)\,$ remains invariant under the reflection transformation with respect to the mirror $\,\Gamma_{k}\,$. Thus $\,A(x)\nabla u(x)\in{\Gamma_{k}}\,$. Since $\,\boldsymbol{n}_k\,$ is the unit normal vector at the mirror $\,\Gamma_{k}\,$, one obtains
\begin{equation}\label{symmetryb}
(A(x)\nabla u)\cdot \boldsymbol{n}_k=0,\qquad \forall\,x\in \Gamma_k
\end{equation}
Equation\,\eqref{symmetryb}\,is the generalized homogeneous\,Neumann\,boundary condition for the second-order linear elliptic boundary value problem.

The governing equation of the second-order elliptic differential boundary value problem on the subdomain $\,\Omega_{1}$ is
\begin{subequations}\label{domain1}
\begin{numcases}{}
    -\nabla\cdot(A(x)\nabla u_{1})+\hat{b}(x)\cdot \nabla u_1+c(x)u_1(x) = f(x) & $x\in \Omega_{1}$ \label{domain1a} \\
    u_{1}(x) = h(x) & $x\in \Gamma_{0}$ \label{domain1b}\\
    (A(x)\nabla u_{1})\cdot\boldsymbol{n}_k = 0 & $x\in \Gamma_{k},~~k\in{\mathcal{I}}$ \label{domain1c}
\end{numcases}
\end{subequations}
where $\,\Gamma_{0}=\partial\Omega_{1}\cap\partial\Omega\,$ is a part of $\,\partial\Omega\,$ inherited from the boundary of the original domain $\,\Omega\,$; for $\,k\in{\mathcal{I}}$, $\Gamma_{k}=\Pi_{k}\cap\partial\Omega_{1}\,$ is a boundary appearing inside $\m{\Omega}$. Boundary condition\,\eqref{domain1b}\,is an essential boundary condition in the finite element method and must be imposed strongly in finite element analysis; boundary condition\,\eqref{domain1c}\,is a natural boundary condition in the finite element method and need not be imposed strongly in finite element analysis.

From the above discussion, it is clear that $\,u|_{\Omega_{1}}=u_{1}$. The solutions $\,u|_{\Omega_{k}}\,$ on the other subdomains $\,\Omega_{k}\,(2\leq k\leq s)\,$ can be transformed into the solution $\,u_{1}$ on the subdomain $\,{\Omega_{1}}$ by several reflection-symmetry transformations. The core of the domain-reduction model is to identify all reflection-symmetry elements contained in the common symmetry group $\,\operatorname{Sym}(\mathcal{L})
\cap\operatorname{Sym}(f,h)\,$. These symmetry elements can then be used to implement the domain-reduction finite element method, thereby significantly reducing the computational effort.

\subsection{Linear Finite Element Discretization}
This section discusses only the linear finite element discretization of the second-order linear elliptic boundary value problem\,(\ref{domain1}). The discretization scheme for problem\,(\ref{domain})\,is completely analogous, differing only in the finite element treatment of the boundary conditions. Since the finite element method is essentially a numerical method based on the variational principle, it is necessary first to give the variational form of\,(\ref{domain1}). The following infinite-dimensional\,Hilbert\,spaces are introduced when defining the variational form:
\begin{gather*}
  L^2(\Omega_1)=\{v:\,\int_{\Omega_1}|v|^2dx<+\infty\} \
  H^{1}(\Omega_1)=\big\{v\in{L^2(\Omega_1)}:\,D^{\alpha}v\in{L^2(\Omega_1)},\,|\alpha|\leq1\big\}
\end{gather*}
where $\,\alpha = (\alpha_1, \dots, \alpha_n)\in \mathbb{N}^n\,$ is a multi-index, $|\alpha| = \alpha_1 + \cdots + \alpha_n$, and $\,D^{\alpha} v\,$ denotes the $\,\alpha\,$th-order weak derivative of $\,v\,$. In particular, when $\,|\alpha| = 0\,$, $D^{\alpha} v = v$; when $\,|\alpha| = 1\,$, $D^{\alpha} v = \partial v / \partial x_i\,$ is a first-order weak derivative of $\,v\,$.

The function space for test functions is $$\,H_{\Gamma_{0}}^1(\Omega_1)=\big\{v\in H^1(\Omega_1):\,v|_{\Gamma_{0}}=0\big\}$$
This is an infinite-dimensional\,Hilbert\,space. Take any test function $\,v\in H_{\Gamma_{0}}^1(\Omega_1)$. Multiplying both sides of equation\,\eqref{domain1a}\,by the test function $\,v$ and integrating over the domain $\,\Omega_1$, one obtains
\begin{equation}\label{fem1}
\int_{\Omega_1}
\left[
-\nabla\cdot(A(x)\nabla u_1)
+\hat{b}\cdot\nabla u_1
+cu_1
\right]v\,dx
=
\int_{\Omega_1} fv\,dx
\end{equation}
Applying the\,Green\,formula to the second-order term in\,(\ref{fem1})\,gives
\begin{equation}\label{fem2}
-\int_{\Omega_1} \nabla\cdot(A(x)\nabla u_{1})v\,dx
=
\int_{\Omega_1} A(x)\nabla u_{1}\cdot\nabla v\,dx
-
\int_{\partial\Omega_1} ((A(x)\nabla u_{1})\cdot \boldsymbol{n})v\,dS
\end{equation}
According to\,\eqref{domain1c}\,and $\,{v(x)=0}$ on the boundary $\m{\Gamma_{0}}$, one has
\begin{equation}\label{femneu2}
\int_{\partial\Omega_1} ((A(x)\nabla u_{1})\cdot \boldsymbol{n})v\,dS=\int_{\Gamma_{0}} ((A(x)\nabla u_{1})\cdot \boldsymbol{n}_0)v\,dS+\sum_{k\in{\mathcal{I}}}\int_{\Gamma_{k}} ((A(x)\nabla u_{1})\cdot \boldsymbol{n}_k)v\,dS=0
\end{equation}
According to\,\eqref{femneu2}\,the boundary integral term in\,(\ref{fem2})\,vanishes. Substituting\,\eqref{fem2}\,into\,\eqref{fem1}
gives the variational form of the boundary value problem\,\eqref{domain1}:

Find $\,u\in \{g(x)\in{H^{1}(\Omega_1)}:\, g(x)=h(x)\,\text{on the boundary}\,\Gamma_{0}\,\text{}\}$ such that
\begin{equation}\label{vara}
a(u_1,v)=F(v),
\qquad
\forall v\in H_{\Gamma_{0}}^1(\Omega_1),
\end{equation}
where the bilinear form $\,a(\cdot,\cdot)\,$ is defined by
\begin{equation*}
a(u_1,v)=\sum_{i,j=1}^{n}
\int_{\Omega_1}
a_{ij}(x)
\frac{\partial u_1}{\partial x_i}
\frac{\partial v}{\partial x_j}
\,dx
+
\int_{\Omega_1} (\hat{b}\cdot\nabla u_1)v\,dx
+
\int_{\Omega_1} cu_{1}v\,dx
\end{equation*}
The linear functional $\,F(\cdot)\,$ is defined by
\begin{equation*}
F(v)=\int_{\Omega_1} fv\,dx
\end{equation*}

Let $\,\mathcal{T}_d\,$ be a regular simplicial triangulation of the $\,n\,$-dimensional bounded domain $\,\Omega_1\,$, where $\,d\,$ denotes the maximum value of all element edge lengths in the triangulation $\,\mathcal{T}_d\,$. Since for $\,n>3\,$ the finite element method on an arbitrary high-dimensional domain inevitably encounters numerical difficulties caused by the curse of dimensionality, the dimension is restricted here to $\,n\leq3$. Take the conforming finite element space $\,V_d\subset H^1(\Omega_1)\,$ as the approximation space of $\,H^1(\Omega_1)\,$. For simplicity, only linear\,Lagrange\,finite elements are used here; thus
\begin{equation*}
V_d=
\left\{
v_d\in C(\overline{\Omega}_1):\,v_d|_K\in P_1(K),\ \forall K\in\mathcal{T}_d
\right\}
\end{equation*}
where $\,P_1(K)\,$ denotes the space of linear polynomials on the element $\,K\,$, and $C(\overline{\Omega}_1)\,$ is the space of continuous functions on $\,\overline{\Omega}_1=\Omega_1\cup\partial\Omega_1\,$.

Let the linear conforming finite element space be
\begin{gather*}
  V_d = \operatorname{span}\{\varphi_1, \varphi_2, \ldots, \varphi_N, \ldots, \varphi_{M+N}\}
\end{gather*}
where\,\(\{\varphi_k\}_{k=1}^{M+N}\)\,is a set of global basis functions of\,\(V_d\), and\,\(M+N = \dim V_d\)\,is the total number of all nodes in the triangulation\,\(\mathcal{T}_d\), including the number of boundary nodes. Denote each node by $\,P_{i}\,(i=1,2,\cdots,M+N)\,$, and let the basis function $\,\varphi_{i}\,$ correspond to the node $\,P_{i}$. To facilitate the treatment of boundary conditions, the basis functions are numbered as follows: \(\{\varphi_k\}_{k=N+1}^{M+N}\)\,are the global basis functions whose nodes $\,\{P_{k}\}_{k=N+1}^{M+N}\,$ lie on the\,Dirichlet\,boundary\,\(\Gamma_0\);
\(\{\varphi_k\}_{k=1}^{N}\)\,are the global basis functions whose nodes $\,\{P_{k}\}_{k=1}^{N}\,$ lie in the interior of the domain $\,\Omega\,$ and on the generalized homogeneous\,Neumann\,boundary $\,\Gamma_{k}\,(k\in{\mathcal{I}})$.

By the definition of the linear conforming finite element space $\,V_{d}\,$, it is easy to obtain the linear finite element subspace approximating $\,H_{\Gamma_{0}}^1(\Omega_1)\,$ as
$$
V_d\cap H_{\Gamma_{0}}^1(\Omega_1) = \operatorname{span}\{\varphi_1, \varphi_2, \ldots, \varphi_N\}
$$
Since $\,u_{1}(x)=h(x)$ on $\,\Gamma_{0}\,$, the linear finite element approximation of $\,u_{1}\,$ is
\begin{equation}\label{uhh}
  u_{1}^{d}(x)=\sum_{k=1}^{N}\xi_{k}\varphi_{k}(x)+\sum_{i=N+1}^{M+N}h(P_{i})\varphi_{i}(x)
\end{equation}
where $\,\xi_{k}\,(k=1,2,\cdots,N)\,$ are unknowns, and $h(P_{i})\,(i=N+1,N+2,\cdots,M+N)\,$ are the known function values of $\,h\,$ at the nodes $\,P_{i}$.

The linear finite element discretization scheme for the variational problem\,\eqref{vara}\,is
\begin{equation}\label{uhdis}
  a\Big(\sum_{k=1}^{N}\xi_{k}\varphi_{k}(x)+\sum_{i=N+1}^{M+N}h(P_{i})\varphi_{i}(x),\varphi_{j}(x)\Big)=F(\varphi_{j}(x)),\quad j=1,\,2,\,\cdots,\,N
\end{equation}
By the bilinearity of $\,a(\cdot,\cdot)\,$, \,\eqref{uhdis}\,can be rewritten as
\begin{equation}\label{uhdisd}
  \sum_{k=1}^{N}a(\varphi_{k}(x),\varphi_{j}(x))\xi_{k}
  =F(\varphi_{j}(x))-\sum_{i=N+1}^{M+N}a(\varphi_{i}(x),\varphi_{j}(x))h(P_{i}),\quad j=1,\,2,\,\cdots,\,N
\end{equation}

Equation\,(\ref{uhdisd})\,can be transformed into the linear system
\begin{equation}\label{linearequ1}
K\xi=F,
\end{equation}
where $\,\xi=(\xi_1,\xi_2,\ldots,\xi_N)^T$, the stiffness matrix is $\,K=(K_{jk})_{N\times N}\,$ with elements $\,K_{jk}=a(\varphi_k(x),\varphi_j(x))$, and the right-hand-side vector is $\,F=(F_j)$ with components $F_j=F(\varphi_{j}(x))-\sum_{i=N+1}^{M+N}a(\varphi_{i}(x),\varphi_{j}(x))h(P_{i})$.

The coefficient matrix in the linear system\,\eqref{linearequ1}\,is large-scale and highly sparse, so fast iterative methods must be used to achieve efficient numerical solution. By numerically solving the linear system\,\eqref{linearequ1}\,one obtains the vector $\,\xi$; substituting it into\,\eqref{uhh} gives the linear finite element approximate solution of the second-order linear elliptic boundary value problem\,\eqref{domain1}.

\section{Dimension-Reducible Mathematical Models and Numerical Examples}
This section mainly verifies the application effect of the solution-symmetry theory in domain-reduction finite element computation. First, a high-dimensional elliptic boundary value problem with a spherically symmetric structure is constructed, and it is shown that it can be reduced from a high-dimensional partial differential equation to a one-dimensional radial ordinary differential equation. Then, a two-dimensional electrostatic field problem, a two-dimensional magnetostatic field problem, and a second-order linear elliptic boundary value problem in which all coefficient functions are nonzero are selected as representative examples. The common symmetry group of each problem and the internal boundary conditions induced by it during domain reduction are first discussed. Subsequently, by comparing the finite element numerical results on the full domain and the reduced domain, it is verified that this reduction method can significantly reduce computational cost while fully preserving the consistency of the structure of the numerical solution.
\subsection{Dimension-Reducible Mathematical Model}
This subsection considers a second-order linear elliptic boundary value problem with a spherically symmetric structure, for which dimensional reduction can be achieved through symmetry analysis. Let
\[
\Omega = \left\{ x \in \mathbb{R}^n : R_1 < \|x\| < R_2 \right\},
\]
where\,\(R_1, R_2 > 0\)\,are given constants and \(n \geq 2\). Consider the following second-order linear elliptic boundary value problem:
\begin{subequations}\label{eq:radial_problem}
\begin{align}
- \nabla \cdot \big( \epsilon(\|x\|) I_n \nabla \varphi \big) &= f(\|x\|), && x \in \Omega, \label{eq:radial_problem_a} \\
\varphi &= C_1, && \|x\| = R_1, \label{eq:radial_problem_b} \\
\varphi &= C_2, && \|x\| = R_2. \label{eq:radial_problem_c}
\end{align}
\end{subequations}
where the coefficient function\,\(\epsilon(\|x\|)\)\,and the source term\,\(f(\|x\|)\)\,both depend only on the radial variable\,\(r = \|x\|\),
and there exists a constant\,\(\alpha > 0\)\,such that\,\(\epsilon(\|x\|)\geq\alpha\)\,holds for any\,\(x \in \Omega\).
$C_1\,$ and $\,C_2\,$ are two given constants, and the unknown function\,\(\varphi\)\,is an \(n\)-variable real-valued function defined on\,\(\Omega\).

By Theorem\,\ref{theorem1}, the boundary value problem\,\eqref{eq:radial_problem}\,is well-posed. Note that $\m{f}$ depends only on $\,r$; hence $\m{f}$ remains invariant under the action of the orthogonal group $\m{O_n}$, namely $\,\operatorname{Sym}(f) = O_n$. Note also that $\,\epsilon(\|x\|)I_n\,$ depends only on\,\(r\). Therefore, $\,Q^{T}\epsilon(\|Qx\|) I_n Q=\epsilon(\|Qx\|) I_n=\epsilon(\|x\|) I_n\,$ holds identically for any $\,Q\in{O_{n}}\,$, and hence $\,\operatorname{Sym}(\epsilon(\|x\|) I_n)=O_{n}$. In addition, the boundary conditions\,\eqref{eq:radial_problem_b}--\eqref{eq:radial_problem_c}\,are likewise invariant with respect to\,\(O_n\). Thus, by Theorem\,\ref{thm:symmetry_solution}, the orthogonal group\,\(O_n\)\,is a subgroup of the symmetry group\,\(\operatorname{Sym}(\varphi)\)\,of the solution\,\(\varphi\). Moreover, \(\operatorname{Sym}(\varphi)\)\,is clearly a subgroup of\,\(O_n\). Therefore, \(\operatorname{Sym}(\varphi) = O_n\). This means that\,\(\varphi\)\,is invariant under the action of the orthogonal transformation group $\,O_{n}\,$, and hence\,\(\varphi\)\,depends only on the radial variable\,\(r = \|x\|\), and is independent of the angular variables. That is, there exists a univariate function\,\(\phi\)\,such that
\begin{equation}\label{varphi}
\varphi(x) = \phi(r)
\end{equation}

Starting from equation\,\eqref{eq:radial_problem_a}, one can derive the ordinary differential equation satisfied by $\,\phi\,$. According to vector identities and the specific expression of the\,Laplace\,operator in an $\,n\,$-dimensional spherical coordinate system, the following identity is obtained:
\begin{equation}\label{identify}
  \nabla \cdot \big( \epsilon(r) \nabla \varphi \big) =\epsilon(r)\Delta \varphi+\nabla\epsilon(r)\cdot \nabla \varphi=
  \epsilon(r)\phi''(r) + \frac{n-1}{r} \epsilon(r)\phi'(r)+\epsilon'(r)\phi'(r).
\end{equation}
Substituting formula\,\eqref{identify}\,into equation\,\eqref{eq:radial_problem_a} gives the ordinary differential equation boundary value problem satisfied by $\,\phi\,$:
\begin{subequations}\label{phiode}
\begin{align}
-\epsilon(r)\phi''(r) - \left(\frac{n-1}{r}\epsilon(r)+\epsilon'(r) \right)\phi'(r) &= f(r),
&& R_{1}<r<R_{2}, \label{phiodea} \\
\phi(R_1) &= C_1, &&\label{phiodeb} \\
\phi(R_2) &= C_2, &&\label{phiodec}
\end{align}
\end{subequations}

By solving the two-point boundary value problem\,\eqref{phiode}\,for the ordinary differential equation, one obtains the solution of the radial function $\,\phi(r)\,$. Substituting it into the transformation relation\,\eqref{varphi}\,then determines the expression of the original unknown function $\,\varphi({x})\,$. When the mathematical expressions of $\,\epsilon(r)\,$ or $\,f(r)\,$ are complicated, analytic solution of equation\,\eqref{phiode}\,is usually very difficult; in this case, numerical methods must be used. The one-dimensional finite element method or one-dimensional spectral element method can both effectively handle this type of problem. Owing to space limitations, this paper does not discuss the implementation details of these numerical algorithms.

In particular, when\,\(\epsilon(r) \equiv 1\), the boundary value problem\,\eqref{phiode}\,can be reduced to the following two-point boundary value problem on the one-dimensional interval\,\((R_1, R_2)\):
\begin{subequations}\label{eq:radial_ode}
\begin{align}
\phi''(r) + \frac{n-1}{r} \phi'(r) &= -f(r), \quad R_1 < r < R_2, \label{eq:radial_ode_a} \\
\phi(R_1) &= C_1, \label{eq:radial_ode_b} \\
\phi(R_2) &= C_2. \label{eq:radial_ode_c}
\end{align}
\end{subequations}

Note the identity
\[
\phi''(r) + \frac{n-1}{r} \phi'(r) = \frac{1}{r^{n-1}} \left( r^{n-1} \phi'(r) \right)',
\]
Therefore, equation\,\eqref{eq:radial_ode_a}\,is equivalent to
\begin{equation}\label{eq:integral_form}
\left( r^{n-1} \phi'(r) \right)' = -r^{n-1} f(r).
\end{equation}

Define
\begin{equation}
F(r) = -\int_{R_1}^{r} \xi^{n-1} f(\xi) \, d\xi.
\end{equation}
Integrating equation\,\eqref{eq:integral_form}\,once gives
\begin{equation}
r^{n-1} \phi'(r) = A + F(r),
\end{equation}
where\,\(A\)\,is an undetermined constant. Therefore,
\begin{equation}
\phi'(r) = r^{1-n} \big( A + F(r) \big).
\end{equation}

Integrating again from\,\(R_1\)\,to\,\(r\)\,and using the boundary condition\,\(\phi(R_1) = C_1\), one obtains
\begin{equation}
\phi(r) = C_1 + A \int_{R_1}^{r} t^{1-n} \, dt - \int_{R_1}^{r} t^{1-n} \left( \int_{R_1}^{t} \xi^{n-1} f(\xi) \, d\xi \right) dt.
\end{equation}

Let
\begin{equation}
I_n(r) = \int_{R_1}^{r} t^{1-n} \, dt,
\end{equation}
and
\begin{equation}\label{eq:J_n_def}
J_n(r) = \int_{R_1}^{r} t^{1-n} \left( \int_{R_1}^{t} \xi^{n-1} f(\xi) \, d\xi \right) dt.
\end{equation}
Then
\begin{equation}
\phi(r) = C_1 + A I_n(r) - J_n(r).
\end{equation}

Using the other boundary condition\,\(\phi(R_2) = C_2\), one obtains
\[
C_2 = C_1 + A I_n(R_2) - J_n(R_2),
\]
and hence
\begin{equation}
A = \frac{C_2 - C_1 +J_n(R_2)}{I_n(R_2)}.
\end{equation}

Therefore, the solution of the boundary value problem\,\eqref{eq:radial_ode}\,can be expressed as
\begin{equation}\label{eq:general_solution}
\phi(r) = C_1 + \frac{C_2 - C_1 +J_n(R_2)}{I_n(R_2)} I_n(r) -J_n(r).
\end{equation}

The explicit expressions for the cases\,\(n \neq 2\)\,and\,\(n = 2\)\,are given respectively below.

\textbf{Case 1: \,\(n \neq 2\).}
In this case,
\[
I_n(r) = \frac{r^{2-n} - R_1^{2-n}}{2 - n},
\]
Substituting into\,\eqref{eq:general_solution}\,gives
\begin{equation}
\phi(r) = C_1 + \frac{C_2 - C_1 + J_n(R_2)}{R_2^{2-n} - R_1^{2-n}} \left( r^{2-n} - R_1^{2-n} \right) - J_n(r),
\end{equation}
where\,\(J_n(r)\)\,is given by\,\eqref{eq:J_n_def}.

\textbf{Case 2: \,\(n = 2\).}
In this case,
\[
I_2(r) = \ln \frac{r}{R_1},
\]
Substituting into\,\eqref{eq:general_solution}\,gives
\begin{equation}
\phi(r) = C_1 + \frac{C_2 - C_1 + J_2(R_2)}{\ln(R_2 / R_1)} \ln \frac{r}{R_1} - J_2(r),
\end{equation}
where\,\(J_2(r)\)\,is defined by\,\eqref{eq:J_n_def}.

\subsection{Two-Dimensional Electrostatic Field Problem}
This subsection considers a two-dimensional multi-material electrostatic field model to verify the applicability of elliptic boundary value problems with reflection-symmetric structures in domain-reduction finite element computation. This numerical model refers to reference\cite{zhang2025}.
The computational domain is the rectangular domain
\begin{equation*}
\Omega=\{(x,y):\,-2<x<2,~-1<y<1\}
\end{equation*}
where the units of $\,x,y\,$ are both\,mm. The material distribution in this electrostatic field model is shown in Fig.\,\ref{elecwq}.
\begin{figure}
    \centering
    \includegraphics[
        height=3.3cm,
        keepaspectratio,
        trim=0 0 0 0,
        clip
    ]{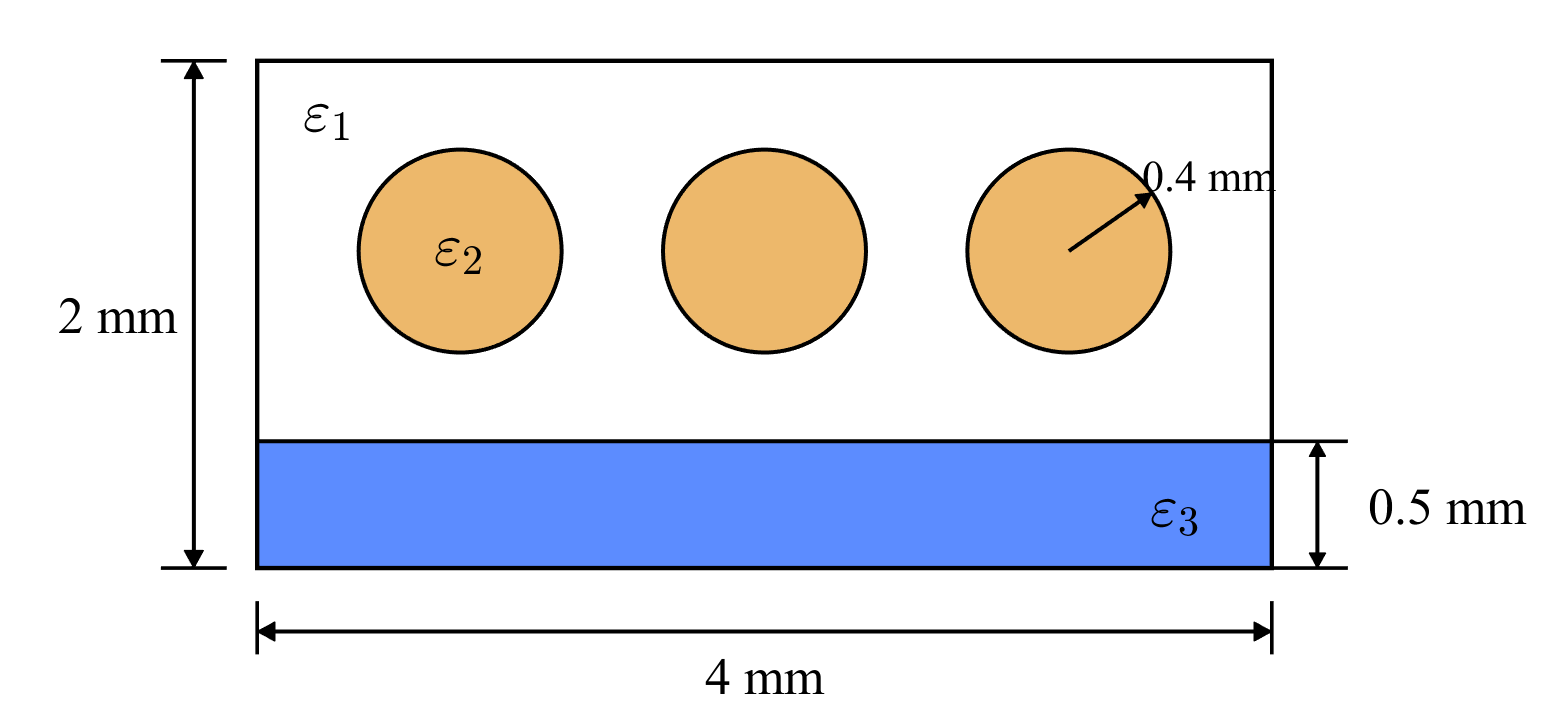}
    \caption{Material distribution of the two-dimensional electrostatic field model}
    \label{elecwq}
\end{figure}
The upper part is a background material with thickness\,1.5\,mm, whose relative permittivity is denoted by $\varepsilon_1$. Three circular dielectric particles with radius\,0.4\,mm are embedded in the background material, and their center coordinates are respectively
\begin{equation*}
C_1=(-1.2,\,0.25)^T,
\qquad
C_2=(0,\,0.25)^T,
\qquad
C_3=(1.2,\,0.25)^T
\end{equation*}
The relative permittivity of the particles is denoted by\,$\varepsilon_2$, and the lower part is a substrate material with thickness\,0.5\,mm, whose relative permittivity is denoted by\,$\varepsilon_3$. Clearly, the material filling the domain is piecewise homogeneous, so the permittivity\,\(\varepsilon(x,y)\)\,is a function of the position coordinates. Under the condition of no free charge, the governing equation for the electric potential in the electrostatic field is
\begin{equation}
-\nabla\cdot\bigl(\varepsilon(x,y)\nabla\phi\bigr)=0,
\qquad (x,y)\in\Omega
\label{eq:dbd_model_eq}
\end{equation}
Here \(\phi(x,y)\) is the electric potential, and $\varepsilon_1=1,\,\varepsilon_2=4, \,\varepsilon_3=0.1$. To form a stable potential difference, a potential of\,\(100\,\mathrm{V}\)\,is applied on the upper boundary of the rectangular domain, and the remaining outer boundaries are grounded. Thus, the mathematical expression of the boundary conditions is
\begin{equation}
\phi\big|_{y=1}=100\,\mathrm{V},
\qquad
\phi\big|_{x=-2}
=
\phi\big|_{x=2}
=
\phi\big|_{y=-1}
=
0\,\mathrm{V}
\label{eq:dbd_model_bc}
\end{equation}

The second-order elliptic boundary value problem\,\eqref{eq:dbd_model_eq}-\eqref{eq:dbd_model_bc} constitutes the physical model of the two-dimensional electrostatic field. Solving this physical model yields the electric potential $\phi$, and the electrostatic field distribution can then be obtained by using $\,\mathbf{E}=-\nabla\phi\,$.
\subsubsection{Symmetry Analysis of the Electrostatic Field Problem}
The equation satisfied by the electric potential,\,\eqref{eq:dbd_model_eq}, is a special case of the second-order linear elliptic boundary value problem\,\eqref{domaina}. The second-order coefficient matrix function can be written as\,$A(X)=\varepsilon(X)I_2,~X=(x,y)^T$, where\,\(I_2\)\,is the second-order identity matrix. The first-order and zeroth-order terms are both zero, namely\,$\hat{b}(x,y)={\bf{0}},\,c(x,y)=0$; the interior source term is $\,f(x,y)=0$, and the boundary function is given by\,\eqref{eq:dbd_model_bc}. The symmetry group\,$\text{Sym}(\Omega)$\,of the rectangular domain is a fourth-order group, namely
$$\text{Sym}(\Omega)=\{I_{2},~Q_{v},~Q_{h},~Q_{v}Q_{h}\}$$
where
\begin{equation}\label{recsym}
I_2=
\begin{pmatrix}
1&0\\
0&1
\end{pmatrix},\quad
Q_v=
\begin{pmatrix}
-1&0\\
0&1
\end{pmatrix},\quad
Q_h=
\begin{pmatrix}
1&0\\
0&-1
\end{pmatrix},\quad
Q_vQ_h=
\begin{pmatrix}
-1&0\\
0&-1
\end{pmatrix}
\end{equation}
Clearly, the transformations corresponding to the matrices $I_{2},~Q_{v},~Q_{h},~Q_{v}Q_{h}\,$ are respectively the identity transformation, the even-symmetry transformation with respect to the $\m{y}$-axis, the even-symmetry transformation with respect to the $\m{x}$-axis, and rotation by $\m{180^{\circ}}$.

Since the material distribution is symmetric with respect to the $\m{y}$-axis but not symmetric with respect to the $\m{x}$-axis, the symmetry group of the relative permittivity is $\,\text{Sym}(\varepsilon)=\text{Sym}(A)=\{I_{2},~Q_v\}$. Since \,\(\hat{b}(x,y)={\bf{0}}\), \(c(x,y)=0\), and \(f(x,y)=0\), the symmetry groups of the first-order term, zeroth-order term, and interior source are all $\,\text{Sym}(\Omega)$. For the boundary conditions, the reflection transformation $\,Q_v\,$ maps the left and right boundaries $\,x=-2\,$ and $\,x=2\,$ to each other, while the upper boundary\,\(y=1\)\,and the lower boundary\,\(y=-1\)\,correspond to themselves, and the electric potential on the boundary remains invariant under this transformation; the other non-identity transformations do not have this property. Therefore, the symmetry group of the boundary function $\,h(x,y)\,$ is $\{I_{2},\,Q_{v}\}$. It follows that the common symmetry group of the coefficient functions, the interior source, and the boundary source in this model is
\[
G=
\operatorname{Sym}(A)\cap
\operatorname{Sym}(\hat{b})\cap
\operatorname{Sym}(c)\cap
\operatorname{Sym}(f)\cap
\operatorname{Sym}(h)
=
\{I_2,\,Q_v\}
\]
According to Theorem\,\ref{thm:symmetry_solution}, the electric-potential function $\,\phi\,$ must inherit the above common symmetry, namely $\,\phi(Q_vX)=\phi(X)$. Written in coordinate form, this is $\,\phi(-x,y)=\phi(x,y),\,\forall(x,y)\in{\Omega}$, which shows that the electric-potential function $\,\phi\,$ is even-symmetric with respect to the $\m{y}$-axis.

Therefore, the electric-potential distribution problem\,\eqref{eq:dbd_model_eq}-\eqref{eq:dbd_model_bc}\,on the whole domain $\m{\Omega}$ can be transformed into the corresponding boundary value problem on the left half-domain
$\Omega_{1/2}=(-2,\,0)\times(-1,\,1)$.
In this case, the boundary condition at $\m{x=0}$ for the elliptic boundary value problem on the half-domain $\,\Omega_{1/2}\,$ is the homogeneous\,Neumann\,boundary condition $\,\frac{\partial \phi}{\partial n}=0$, while the boundary conditions on the other boundaries of the subdomain problem inherit the\,Dirichlet\,boundary conditions on the whole domain $\m{\Omega}$.

\subsubsection{Finite Element Numerical Simulation of the Electrostatic Field Problem}

In this example, linear triangular finite elements are used to perform numerical simulations on the full-domain and half-domain models, respectively. The full-domain computation result is used to show the overall electric-potential distribution in the multi-material structure, and the half-domain computation result is used to verify whether the symmetry-based domain-reduction method can reproduce the solution structure of the corresponding part of the full domain. The numerical computation results are shown in Fig.\,\ref{fig:electrostatic_results}.
\begin{figure}
    \centering
    \begin{minipage}[t]{0.50\textwidth}
        \centering
        \includegraphics[
            height=3.2cm,
            keepaspectratio,
            trim=0 0 0 0,
            clip
        ]{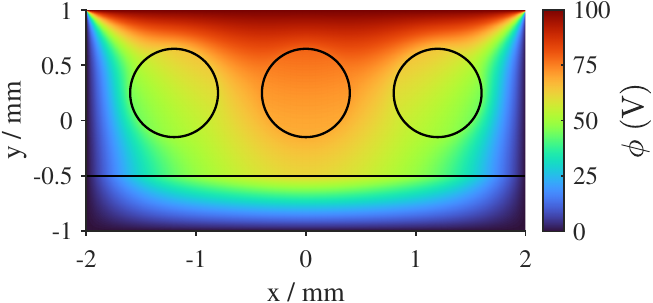}

        \vspace{0.2em}
        (a) Electric-potential distribution on the full domain
    \end{minipage}
    \hspace{0.04\textwidth}
    \begin{minipage}[t]{0.34\textwidth}
        \centering
        \includegraphics[
            height=3.2cm,
            keepaspectratio,
            trim=0 0 0 0,
            clip
        ]{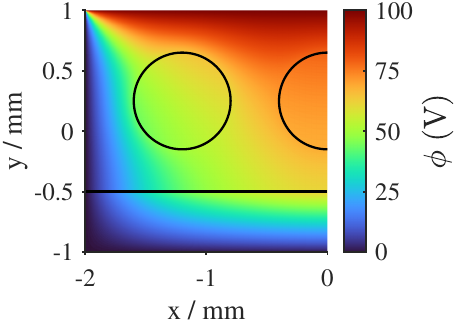}

        \vspace{0.2em}
        (b) Electric-potential distribution on the left half-domain
    \end{minipage}

    \vspace{0.3em}
    \caption{Electric-potential distributions on the full domain and the half-domain}
    \label{fig:electrostatic_results}
\end{figure}

Fig.\,\ref{fig:electrostatic_results}\,shows that the electric-potential distribution in the full domain $\m{\Omega}$ exhibits a clear mirror-symmetry feature with respect to the $\m{y}$-axis, and that the half-domain computation result can reproduce well the electric-potential distribution of the corresponding part of the full domain. Therefore, when the material distribution and boundary conditions have common reflection symmetry, symmetry can be used to reduce the full domain to a half-domain for finite element solution. This method reduces the computational domain and the number of degrees of freedom while preserving the consistency of the electric-potential distribution structure, thereby improving numerical computational efficiency.

\subsection{Two-Dimensional Magnetostatic Field Problem}
This subsection considers the two-dimensional magnetostatic field distribution generated by two parallel current-carrying conductors in a homogeneous medium. It is assumed that the conductors extend infinitely in the $\,z\,$ direction and that the field quantities do not vary in the $\,z\,$ direction. Therefore, a two-dimensional computational model can be established in the\,\(xOy\)\,plane. The computational domain is taken as the rectangular domain
\begin{equation*}
\Omega=\{(x,y):\,-20<x<20,~-10<y<10\}
\end{equation*}
where the units of $\,x,y\,$ are both\,mm.

Assume that the medium filling the domain is air, with relative permeability $\,\mu_r=1$. Two circular conductors are placed in the rectangular domain, and the radii of their cross sections are both $\,1\,\mathrm{mm}$.
The centers of the two conductors lie on the same horizontal line, and their coordinates are respectively
\begin{equation*}
X_1=(-10,\,0)^{T},
\qquad
X_2=(10,\,0)^{T}.
\end{equation*}
The model structure is shown in Fig.\,\ref{fig:mag_structure_model}.

\begin{figure}
    \centering
    \includegraphics[
        height=3.3cm,
        keepaspectratio,
        trim=0 0 0 0,
        clip
    ]{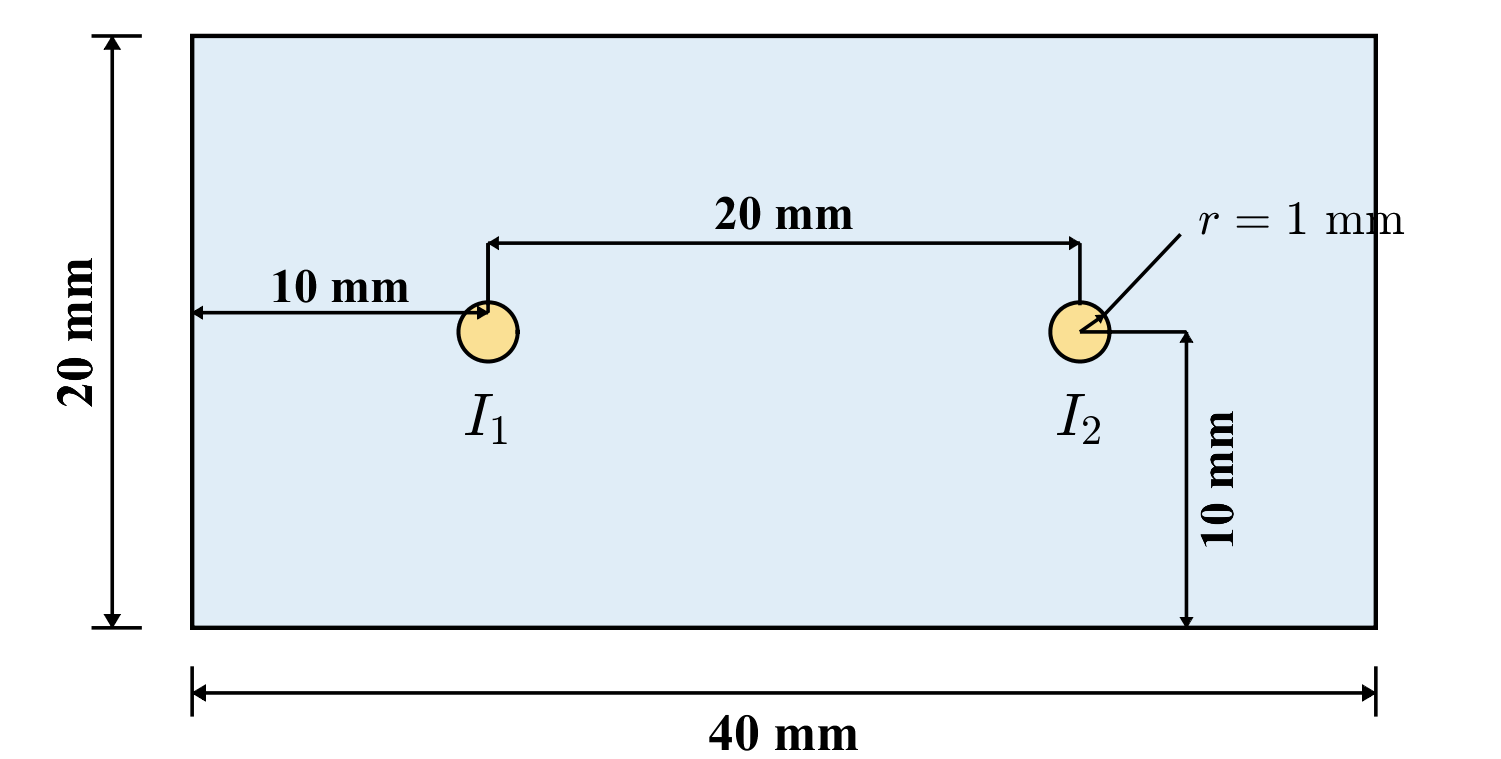}
    \caption{Structural schematic of the two-dimensional magnetostatic field model}
    \label{fig:mag_structure_model}
\end{figure}

The currents in the left and right conductors are denoted by\,\(I_1\) and\,\(I_2\), respectively. In the computation, the two conductor currents are set to have equal magnitudes and the same direction, namely
\begin{equation*}
I_1=100\,\mathrm{A},
\qquad
I_2=100\,\mathrm{A}.
\end{equation*}
Here, \(I_1\)\,and\,\(I_2\)\,both denote currents in the\,\(+z\)\,direction. In the two-dimensional magnetostatic field problem, the magnetic vector potential has only a component in the\,\(z\)\,direction, denoted by $\,\mathbf{A}=(0,0,A_z)$. The magnetic flux density is given by the curl of the magnetic vector potential, namely $\,\mathbf{B}=\nabla\times\mathbf{A}$. Therefore, in the two-dimensional case,
\begin{equation}
B_x=\frac{\partial A_z}{\partial y},
\qquad
B_y=-\frac{\partial A_z}{\partial x}.
\label{eq:wire_B_A}
\end{equation}

Let $\,\nu=\mu^{-1},\,\mu=\mu_0\mu_r$, where $\mu_{0}$ is the magnetic permeability constant in vacuum. Since the medium in this model is homogeneous,\,\(\nu\)\,is a constant function on $\,\Omega\,$. The magnetic vector potential\,\(A_z\)\,satisfies the following elliptic boundary value problem:
\begin{equation}
-\nabla\cdot(\nu\nabla A_z)=J_z,
\qquad (x,y)\in\Omega
\label{eq:wire_model_eq}
\end{equation}
The current inside each conductor is taken to be uniformly distributed, so the magnitude of the current density is $\,J_0=100\pi^{-1}\,\text{A}\cdot\mathrm{mm}^{-2}$. Denote the cross-sectional regions of the left and right conductors by\,\(D_1\)\,and\,\(D_2\), respectively. Then the source term can be written as
\begin{equation*}
J_z(x,y)=
\begin{cases}
 J_0, & (x,y)\in D_1\cup D_2\\
0, & (x,y)\in \Omega\setminus(D_1\cup D_2)
\end{cases}
\end{equation*}
The zero magnetic-vector-potential boundary condition is adopted on the outer boundary:
\begin{equation}
A_z(x,y)=0,
\qquad (x,y)\in\partial\Omega
\label{eq:wire_model_bc}
\end{equation}

The second-order elliptic boundary value problem\,\eqref{eq:wire_model_eq}-\eqref{eq:wire_model_bc}\,constitutes the physical model of the two-dimensional magnetostatic field. Solving this model yields $\,A_{z}$, and substituting it into\,\eqref{eq:wire_B_A}\,gives the magnetic-flux-density distribution.
\subsubsection{Symmetry Analysis of the Magnetostatic Field Problem}
The equation\,\eqref{eq:wire_model_eq}-\eqref{eq:wire_model_bc}\,satisfied by the component $\m{A_{z}}$ of the magnetic vector potential is a special case of the second-order linear elliptic\,Dirichlet\,boundary value problem\,\eqref{domain}. Here, the second-order coefficient matrix function is $\,A(x,y)=\nu I_2$, the first-order column-vector function is $\,\hat{b}(x,y)={\bf{0}}\,$, the zeroth-order term is $\,c(x,y)=0$, the interior source term $\,f(x,y)\,$ is the current-density function\,\(J_z(x,y)\), and the boundary source function is $\,h(x,y)=0$.

The symmetry group of the rectangular domain $\,\Omega\,$ is $\,\operatorname{Sym}(\Omega)=\{I_{2},\,Q_h,\,Q_v,\,Q_vQ_{h}\}$, where the matrices $I_{2},\,Q_h,\,Q_v,\,Q_vQ_{h}$ are defined by\,\eqref{recsym}. Since the medium is homogeneous,\,\(\nu\)\,is a constant function, and therefore the symmetry group of the second-order coefficient matrix function\,\(A(x,y)=\nu I_2\)\,is
$\text{Sym}(A)=\text{Sym}(\nu)=\text{Sym}(\Omega)$. Since\,\(\hat{b}={\bf{0}}\) and \(c=0\),
$
\operatorname{Sym}(\hat{b})=
\operatorname{Sym}(c)=
\operatorname{Sym}(\Omega)
$.

In this example, the two conductor currents have equal magnitudes, the same direction, and a symmetric distribution; therefore the symmetry group of the interior source term $J_{z}(x,y)$ is $\,\operatorname{Sym}(J_z)=\operatorname{Sym}(\Omega)$. In addition, the outer boundary condition is a homogeneous\,Dirichlet\,condition, and therefore the symmetry group of the boundary source function is $\,\operatorname{Sym}(h)=\operatorname{Sym}(\Omega)$.

In summary, the common symmetry group of this physical model is
\begin{equation*}
G=
\operatorname{Sym}(A)
\cap \operatorname{Sym}(\hat{b})
\cap \operatorname{Sym}(c)
\cap \operatorname{Sym}(J_z)
\cap \operatorname{Sym}(h)
=\operatorname{Sym}(\Omega)=
\{I,Q_h,Q_v,Q_hQ_v\}
\end{equation*}
According to Theorem\,\ref{thm:symmetry_solution}, the component\,\(A_z\)\,of the magnetic vector potential inherits the above common symmetry, namely $\,G\subset\text{Sym}(A_{z})$, and hence
$$A_{z}(x,y)=A_{z}(-x,y)=A_{z}(x,-y)=A_{z}(-x,-y),~\forall(x,y)\in{\Omega}$$

According to the domain-reduction finite element theory, the second-order elliptic boundary value problem\,\eqref{eq:wire_model_eq}-\eqref{eq:wire_model_bc}\,on the whole rectangular domain $\m{\Omega}$ can be reduced, by using the symmetry of the solution, to the corresponding boundary value problem on a\,1/4\,domain. The lower-left\,1/4\,domain $\,\Omega_{1/4}=(-20,0)\times(-10,0)\,$ is selected as the computational subdomain. The boundary of this subdomain consists of two types of boundary conditions: the homogeneous\,Dirichlet\,condition inherited on the original boundary, and the homogeneous\,Neumann\,condition imposed on the newly added symmetry boundaries $\,x=0\,$ and $\,y=0\,$.

\subsubsection{Finite Element Numerical Simulation of the Magnetostatic Field Problem}

Linear triangular finite elements are used to numerically solve the full-domain and\,1/4\,domain models, respectively. The full-domain computation result is used to show the overall distribution of the magnetic-vector-potential component in the rectangular domain, and the\,1/4\,domain computation result is used to verify whether the symmetry-based domain-reduction method can reproduce the solution structure of the corresponding part of the full domain. The numerical computation results are shown in Fig.\,\ref{fig:magnetostatic_results}.

\begin{figure}
    \begin{center}
    \begin{minipage}[t]{0.4\textwidth}
        \centering
        \includegraphics[
            height=3.2cm,
            keepaspectratio,
            trim=0 0 0 0,
            clip
        ]{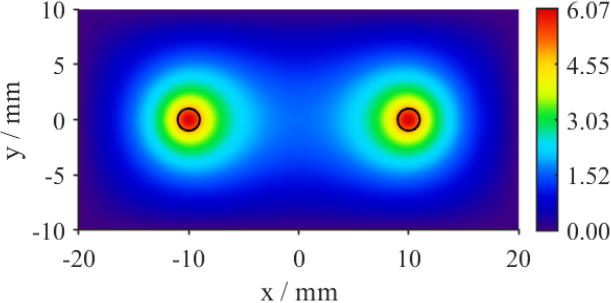}

        \vspace{0.4em}
        (a) Numerical result on the full domain
    \end{minipage}
    \hspace{0.04\textwidth}
    \begin{minipage}[t]{0.32\textwidth}
        \centering
        \includegraphics[
            height=3.2cm,
            keepaspectratio,
            trim=0 0 0 0,
            clip
        ]{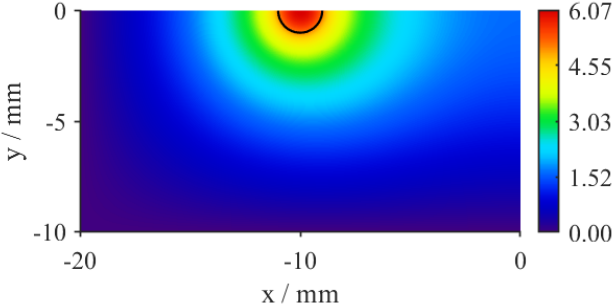}
        \vspace{0.2em}
        (b) Numerical result on the lower-left subdomain
    \end{minipage}
\end{center}
    \vspace{0.3em}
    \caption{Distribution of the magnetic-vector-potential component on the full domain and the half-domain}
    \label{fig:magnetostatic_results}
\end{figure}

As shown in Fig.\,\ref{fig:magnetostatic_results}, the distribution of the magnetic-vector-potential component in the full domain $\m{\Omega}$ has mirror symmetry with respect to both the $\m{x}$-axis and the $\m{y}$-axis; this symmetry implies central-inversion invariance, and the computation result on the lower-left\,1/4\,domain is consistent with the magnetic-vector-potential component distribution at the corresponding location in the full domain. Based on the mirror symmetry of the model, the full domain can be further reduced to a\,1/4\,domain for finite element solution, thereby significantly reducing the computational effort.

\subsection{Second-Order Linear Elliptic Boundary Value Problem with All Coefficient Functions Nonzero}
A second-order linear elliptic boundary value problem on a circular domain is selected as the most general numerical example to verify the effectiveness of the solution-symmetry theory and the domain-reduction finite element method. The computational domain is a disk centered at the origin with radius\,\(R=2\,\mathrm{m}\):
\begin{equation*}
\Omega=\left\{(x,y)\in\mathbb{R}^{2}:x^{2}+y^{2}<R^{2}\right\}
\end{equation*}
where the units of $\,x,y\,$ are both\,m.
The second-order linear elliptic\,Dirichlet\,boundary value problem\,\eqref{domain}\,is the computational model in this section. The second-order coefficient matrix function, first-order coefficient column-vector function, and zeroth-order coefficient function are respectively taken as
\begin{equation*}
A(x,y)=
\begin{pmatrix}
2 & 1\\
1 & 2
\end{pmatrix},
\qquad
\hat{b}(x,y)=
\begin{pmatrix}
2x-y\\
2y-x
\end{pmatrix},
\qquad
c(x,y)=x+y
\end{equation*}
\begin{figure}
    \centering
    \includegraphics[
        height=4.2cm,
        keepaspectratio,
        trim=0 0 0 0,
        clip
    ]{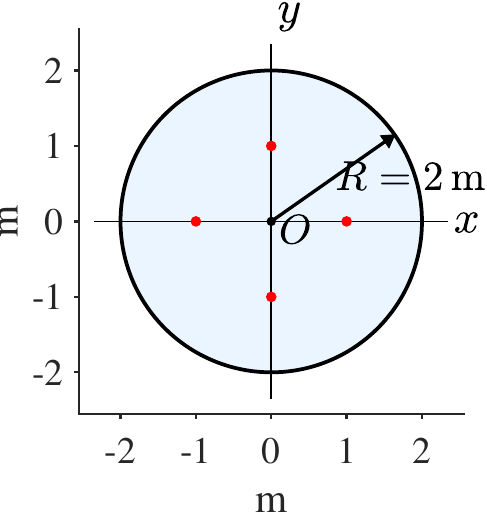}
    \caption{Geometric model of the second-order linear elliptic boundary value problem}
    \label{fig:elc_structure_mode2}
\end{figure}
The interior source function and boundary source function are taken respectively as
\begin{equation*}
  f(x,y)=\sum_{k=1}^{4}q_{k}\delta(x-x_{k},y-y_{k}),\quad h(x,y)=(x+y)^2
\end{equation*}
In the above expression,\,\(\delta\)\,is the Dirac\,delta\,function, and the positions and source strengths of the four point sources are respectively
\begin{equation*}
(x_{1},y_{1})=(1,0)^{T},\quad
(x_{2},y_{2})=(0,1)^{T},\quad
(x_{3},y_{3})=(-1,0)^{T},\quad
(x_{4},y_{4})=(0,-1)^{T},
\end{equation*}
\begin{equation*}
q_{1}=10,\quad q_{2}=10,\quad q_{3}=1,\quad q_{4}=1
\end{equation*}
The computational model is shown in Fig.\,\ref{fig:elc_structure_mode2}.

\subsubsection{Symmetry Analysis of the Model Problem}

To simplify the symmetry analysis, let
\begin{equation*}
X=
\begin{pmatrix}
x\\
y
\end{pmatrix},
\qquad
\hat{b}(X)=BX,
\qquad
B=
\begin{pmatrix}
2 & -1\\
-1 & 2
\end{pmatrix}
\end{equation*}
At the same time, the zeroth-order coefficient function and boundary function can be written as
\begin{equation*}
c(X)=l^{T}X,
\qquad
h(X)=(l^{T}X)^{2},
\qquad
l=
\begin{pmatrix}
1\\
1
\end{pmatrix}
\end{equation*}

Clearly, the symmetry group $\,\operatorname{Sym}(\Omega)\,$ of the circular domain is the second-order orthogonal group $\,O(2)$. For the second-order coefficient matrix\,\(A\), its symmetry condition is that\,\(Q^{T}AQ=A\)\,holds for $\,\forall\,Q\in{O_2}\,$. The condition $Q^{T}AQ=A$\,is equivalent to $AQ=QA$, and hence one obtains
\begin{equation*}
\operatorname{Sym}(A)=\{I_2,\,P,\,-I_2,\,-P\}
\end{equation*}
where
\begin{equation*}
I_2=
\begin{pmatrix}
1 & 0\\
0 & 1
\end{pmatrix},
\quad
P=
\begin{pmatrix}
0 & 1\\
1 & 0
\end{pmatrix}
\end{equation*}
Here,\,\(I_{2}\)\,denotes the identity transformation, \(P\)\,denotes the reflection transformation with respect to the line\,\(y=x\), \(-P\)\,denotes the reflection transformation with respect to the line\,\(y=-x\), and \(-I_{2}\)\,denotes the central-inversion transformation.

For the first-order coefficient column-vector function\,\(\hat{b}(X)=BX\), its symmetry condition is $\,Q^{T}\hat{b}(QX)=\hat{b}(X),\,\forall\,Q\in{O_2}$, which means $\,\hat{b}(QX)=Q\hat{b}(X),\,\forall\,Q\in{O_2}$. Thus $\,BQ=QB,\,\forall\,Q\in{O_2}$, and consequently
\begin{equation*}
\operatorname{Sym}(\hat{b})=\{I_2,\,P,\,-I_2,\,-P\}
\end{equation*}

For the zeroth-order coefficient function\,\(c(X)=l^{T}X\), its symmetry condition is\,\(c(QX)=c(X)\),
which is equivalent to\,\(Q^{T}l=l\). Therefore,
\begin{equation*}
\operatorname{Sym}(c)=\{I_2,P\}
\end{equation*}

For the boundary function\,\(h(X)=(l^{T}X)^{2}\), its symmetry condition is\,\(h(QX)=h(X)\),
which is equivalent to\,\(Q^{T}l=l\)\,or\,\(Q^{T}l=-l\). Therefore,
\begin{equation*}
\operatorname{Sym}(h)=\{I_{2},\,P,\,-I_{2},\,-P\}
\end{equation*}

This model contains four point sources, and the correspondence between their positions and source strengths is
\[
((x_1,y_1),q_1)=((1,0)^T,\,10),\quad
((x_2,y_2),q_2)=((0,1)^T,\,10)\]
\[((x_3,y_3),q_3)=((-1,0)^T,\,1),\quad
((x_4,y_4),q_4)=((0,-1)^T,\,1)
\]
Under the action of the reflection transformation\,\(P\), the interior source term maps \((1,0)\)\,and\,\((0,1)\)\,to each other, and maps \((-1,0)\)\,and\,\((0,-1)\)\,to each other, while the corresponding point-source strengths remain the same. Under the action of central inversion\,\(-I_2\)\,or the reflection transformation\,\(-P\), however, the point source with strength\,\(10\)\,is mapped to the position of a point source with strength\,\(1\), so the correspondence of source strengths is no longer preserved. Therefore, the symmetry group of the interior source is
\begin{equation*}
\operatorname{Sym}(f)=\{I_2,\,P\}
\end{equation*}

In summary, the common symmetry group of the equation coefficients, interior source, and boundary source is
\begin{equation*}
G=\operatorname{Sym}(A)\cap
\operatorname{Sym}(\hat{b})\cap
\operatorname{Sym}(c)\cap
\operatorname{Sym}(f)\cap
\operatorname{Sym}(h)
=
\{I_2,\,P\}
\end{equation*}
By Theorem\,\ref{thm:symmetry_solution}, the unique solution $\,u(x,y)\,$ of this mathematical model satisfies
\begin{equation*}
u(x,y)=u(y,x),~\forall(x,y)\in{\Omega}
\end{equation*}
This shows that the solution $\m{u}$ of the second-order linear elliptic differential equation is symmetric with respect to the line $\m{y=x}$. Therefore, the second-order linear elliptic boundary value problem on the full disk can be reduced to the corresponding boundary value problem on a half-domain. In this case, the generalized homogeneous\,Neumann\,boundary condition $A\nabla u\cdot n=0$ must be used on the internal symmetry boundary $\m{y=x}$.

\subsubsection{Finite Element Numerical Simulation of the Model Problem}
Linear triangular finite elements are used to discretize and solve the full domain and the reduced half-domain separately. The full-domain finite element computation result is used to show the overall solution distribution of the original boundary value problem on the disk, and the half-domain finite element computation result is used to verify whether the domain-reduction model can reproduce the solution structure of the corresponding part of the full domain. The two computations use the same equation coefficients, point-source settings, and outer-boundary\,Dirichlet\,conditions; the difference is that the half-domain model uses a homogeneous generalized\,Neumann\,boundary condition on the internal symmetry boundary\,\(y=x\). The finite element numerical solutions are shown in Fig.\,\ref{fig:elc_two_results}.\\
\begin{figure}
    \centering
    \begin{minipage}[b]{0.47\textwidth}
        \centering
        \includegraphics[
            height=3.6cm,
            keepaspectratio,
            trim=0 0 0 0,
            clip
        ]{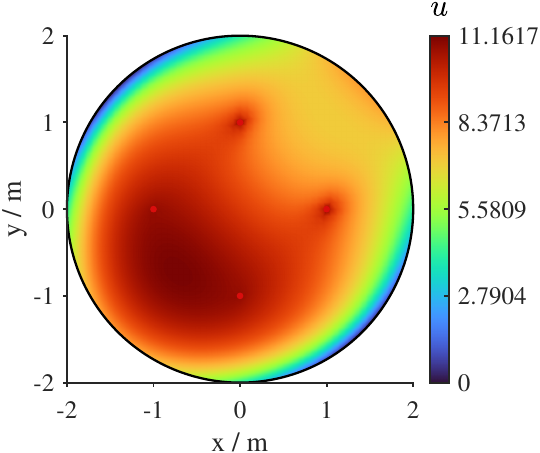}

        \vspace{0.2em}
        (a) Numerical result on the full disk
    \end{minipage}
    \hfill
    \begin{minipage}[b]{0.47\textwidth}
        \centering
        \includegraphics[
            height=3.6cm,
            keepaspectratio,
            trim=0 0 0 0,
            clip
        ]{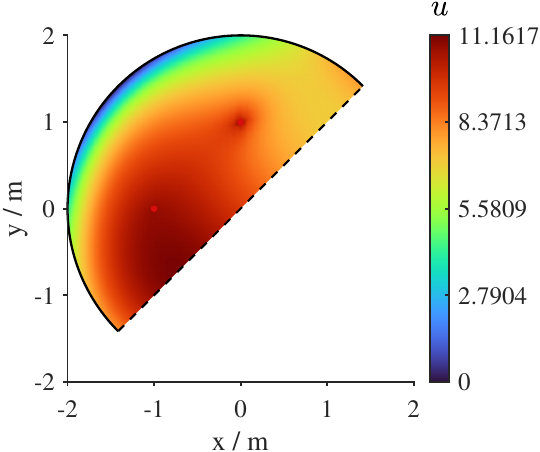}

        \vspace{0.2em}
        (b) Numerical result on the symmetric half-domain
    \end{minipage}

    \vspace{0.3em}
    \caption{Comparison of finite element numerical solutions on the full disk and the symmetric half-domain}
    \label{fig:elc_two_results}
\end{figure}
\indent As shown in Fig.\,\ref{fig:elc_two_results}, the finite element numerical solution on the full disk exhibits a clear mirror-symmetry feature with respect to the line $\m{y=x}$, and the numerical solution obtained from the half-domain model is consistent with the solution distribution at the corresponding location in the full domain. This example verifies the correctness of the theory established in this paper and also shows that exploiting the intrinsic symmetry of the problem can effectively reduce the computational scale and improve solution efficiency.
\section{Conclusions and Outlook}
This paper studied the symmetry of solutions to second-order linear elliptic\,Dirichlet\,boundary value problems on bounded domains in $\m{n}$ dimensions. The symmetry groups of the second-order coefficient matrix function, first-order coefficient column-vector function, zeroth-order coefficient function, interior source, and boundary source were defined respectively. On this basis, it was rigorously proved that the common symmetry group of the above five types of symmetry groups is a subgroup of the symmetry group of the solution. If the common symmetry group contains several reflection-symmetry elements, then the boundary value problem for the second-order linear elliptic equation on the original full domain can be reduced to the corresponding problem on a subdomain. It was rigorously proved that the boundary condition arising on the boundary of the subdomain is a homogeneous generalized\,Neumann\,boundary condition. For the case $\m{n=2}$, the elliptic boundary value problem on the subdomain was solved numerically by the linear finite element method. Finally, the theoretical example and numerical experiments presented in this paper verified the theoretical results.\\
\indent This paper established a symmetry-reduction theory and numerical verification for second-order linear elliptic\,Dirichlet\,problems on bounded domains. Future research can be extended in three directions. Theoretically, the symmetry-group analysis method can be generalized to elliptic equations on unbounded domains, the influence of boundary conditions at infinity on the symmetry-group structure and symmetry reduction of solutions can be investigated, and the symmetry-reduction theory for unbounded-domain problems can be further improved. In terms of boundary conditions, the study can be extended to more general cases such as\,Neumann, Robin, and mixed boundary conditions, so as to reveal the laws of symmetry reduction of solutions and the forms of subdomain boundary conditions under different boundary conditions. In terms of applications, this symmetry-reduction method can substantially reduce the solution complexity of high-dimensional problems and has important application value in engineering fields such as steady-state heat conduction, electrostatic-field analysis, and seepage mechanics. Future work may combine this method with practical engineering models to develop efficient numerical algorithms based on symmetry and provide a new approach for solving complex engineering problems.


\end{document}

%% file: latex_defs.tex
\def\^{\hat}
\def\~{\tilde}
\def\3h{{3\over 2}}

\def\eqn#1$${\eqno{{\rm #1}}$$}

\def\~{\tilde}

\def\^{\hat}

\def\XXint#1#2#3{{\setbox0=\hbox{$#1{#2#3}{\int}$}
     \vcenter{\hbox{$#2#3$}}\kern-.5\wd0}}


%% file: waveguidecondition.bbl
\begin{thebibliography}{0}%
\makeatletter
\providecommand \@ifxundefined [1]{%
 \@ifx{#1\undefined}
}%
\providecommand \@ifnum [1]{%
 \ifnum #1\expandafter \@firstoftwo
 \else \expandafter \@secondoftwo
 \fi
}%
\providecommand \@ifx [1]{%
 \ifx #1\expandafter \@firstoftwo
 \else \expandafter \@secondoftwo
 \fi
}%
\providecommand \natexlab [1]{#1}%
\providecommand \enquote  [1]{``#1''}%
\providecommand \bibnamefont  [1]{#1}%
\providecommand \bibfnamefont [1]{#1}%
\providecommand \citenamefont [1]{#1}%
\providecommand \href@noop [0]{\@secondoftwo}%
\providecommand \href [0]{\begingroup \@sanitize@url \@href}%
\providecommand \@href[1]{\@@startlink{#1}\@@href}%
\providecommand \@@href[1]{\endgroup#1\@@endlink}%
\providecommand \@sanitize@url [0]{\catcode `\\12\catcode `\$12\catcode
  `\&12\catcode `\#12\catcode `\^12\catcode `\_12\catcode `\%12\relax}%
\providecommand \@@startlink[1]{}%
\providecommand \@@endlink[0]{}%
\providecommand \url  [0]{\begingroup\@sanitize@url \@url }%
\providecommand \@url [1]{\endgroup\@href {#1}{\urlprefix }}%
\providecommand \urlprefix  [0]{URL }%
\providecommand \Eprint [0]{\href }%
\providecommand \doibase [0]{http://dx.doi.org/}%
\providecommand \selectlanguage [0]{\@gobble}%
\providecommand \bibinfo  [0]{\@secondoftwo}%
\providecommand \bibfield  [0]{\@secondoftwo}%
\providecommand \translation [1]{[#1]}%
\providecommand \BibitemOpen [0]{}%
\providecommand \bibitemStop [0]{}%
\providecommand \bibitemNoStop [0]{.\EOS\space}%
\providecommand \EOS [0]{\spacefactor3000\relax}%
\providecommand \BibitemShut  [1]{\csname bibitem#1\endcsname}%
\let\auto@bib@innerbib\@empty
\end{thebibliography}%


\begin{thebibliography}{10}
\bibitem{zhou2005}
Zhou S X. Partial Differential Equation[M]. Beijing: Peking University Press, 2005. (In Chinese)

\bibitem{Ovsi1982}
Ovsiannikov L V. Group Analysis of Differential Equations[M]. Academic Press, 1982.

\bibitem{Olver1993}
 Olver P J. Applications of Lie Groups to Differential Equations[M]. Springer, 1993.

\bibitem{GIDAS1979} Gidas B, Ni W and Nirenberg L. Symmetry and related properties via the maximum principle[J]. Communications in Mathematical Physics, 1979, 68(3): 209-243.

\bibitem{FRA2000}
Fraenkel L E. An Introduction to Maximum Principles and Symmetry in Elliptic Problems[M]. Cambridge: Cambridge University Press, 2000.

\bibitem{bal1982}
Ballisti R, Hafner C and Leuchtmann P. Application of the representation theory of finite groups to field computation problems with symmetrical boundaries[J], IEEE Transactions on Magnetics, 1982, 18(2):584--587.


\bibitem{dou192}
Douglas C C and Mandel J. An abstract theory for the domain reduction method[J]. Computing, 1992, 48(1): 73-96.

\bibitem{all1992}
Allgower E L, B\"ohmer K, Georg K and Miranda R. Exploiting symmetry in boundary element methods[J],
SIAM Journal on Numerical Analysis, 1992, 29(2):534--552.

\bibitem{bos1986}
Bossavit A. Symmetry, groups, and boundary value problems. A progressive introduction to noncommutative harmonic analysis of partial differential equations in domains with geometrical symmetry[J]. Computer Methods in Applied Mechanics and Engineering, 1986, 56(2): 167-215.

\bibitem{bos1993}
Bossavit A. Boundary value problems with symmetry and their approximation by finite elements[J]. SIAM Journal on Applied Mathematics, 1993, 53(5): 1352-1380.

\bibitem{lob1994}
Lobry J and Broche C. Exploitation of the geometrical symmetry in the boundary element method with the group representation theory[J], IEEE Transactions on Magnetics, 1994, 30(1):118--123.

\bibitem{lob1996}
Lobry J. Use of group theory in symmetrical 3-D eddy-current problems[J], IEE Proceedings - Science, Measurement and Technology, 1996, 143(6): 369--376.


\bibitem{hou2022}
Hou P, Liu F and Zhou A, Symmetrized two-scale finite element discretizations for partial differential equations with symmetric solutions. arXiv:2205.15524, 2022.

\bibitem{wang2023}
Wang C, Peng R, He Y, Yang H and Han X. A scaled boundary finite element partitioning based reduced order algorithm for the elastic analysis of cyclically symmetric structures[J]. International Journal for Numerical Methods in Engineering, 2023.

\bibitem{wang2024}
Wang J, Liu L and Chen Y. Efficient finite element modeling of photonic structures with combined symmetry operations for waveguide modal analysis[C], 2024 Light Conference, Changchun, China, 2024, pp. 1-5.

\bibitem{wang2015}
Wang D M. Numerical Methods for Elliptic Partial Differential Equations[M]. Beijing: Science Press, 2015. (In Chinese)

\bibitem{bre2008} Brenner S C and Scott L R. The Mathematical Theory of Finite Element Methods[M]. 3rd ed. New York: Springer, 2008.

\bibitem{Ciar2002}
Ciarlet P G. The Finite Element Method for Elliptic Problems[M]. Philadelphia: Society for Industrial and Applied Mathematics, 2002.

\bibitem{chenshao2020}
Zhao Z J and Chen S C. A prism element for a second-order elliptic mixed problem[J]. Acta Mathematica Scientia, 2020, 40(3): 684-693. (In Chinese)

\bibitem{Dol2015} Dolbeault J, Felmer P and Monneau R. Symmetry and nonuniformly elliptic operators[J]. Differential and Integral Equations, 2005, 18(2): 141-154.

\bibitem{Conway}
Conway J B. A Course in Functional Analysis[M]. New York: Springer, 1990.

\bibitem{Evans}
Evans L C. Partial Differential Equations[M]. American Mathematical Society, 2022.

\bibitem{sitikejin}
Kostrikin A I. Introduction to Algebra (Vol. I): Basic Algebra[M]. Trans. Zhang, Y. B. 2nd ed. Beijing: Higher Education Press, 2006. (In Chinese)
\bibitem{artin2014}
Artin M. Algebra[M]. Trans. Yao H L and Ping Y R. 2nd ed. Beijing: China Machine Press, 2014. (In Chinese)
\bibitem{bur2001}
Burago D, Burago Y and Ivanov S. Metric Geometry[M]. American Mathematical Society, 2001.

\bibitem{zhang2025}
    Zhang Y J, Zhang Q, Sun H D, Zhao Z T and Huang X F. Electromagnetic field computation method based on deep Ritz method[J]. Transactions of China Electrotechnical Society, 2025, 40(23): 7462-7474. (In Chinese)
\end{thebibliography}
